\documentclass[11pt]{amsart}
\usepackage[numbers, square]{natbib}
\usepackage{amssymb}
\usepackage{amsmath}
\usepackage{amsfonts}
\usepackage[hmarginratio={1:1},     
  vmarginratio={1:1},     
  textwidth=430pt,        
  heightrounded]{geometry}
\usepackage{bbm}
\usepackage{dsfont}
\usepackage{graphicx}
\usepackage{amsthm}
\usepackage{hyperref}
\usepackage[dvipsnames]{xcolor}
\usepackage{esint}

\theoremstyle{plain}

\newtheorem{corollary}{Corollary}

\newtheorem{example}{Example}

\newtheorem{lemma}{Lemma}

\newtheorem{theorem}{Theorem}
\theoremstyle{remark}
\newtheorem{remark}{Remark}
\numberwithin{equation}{section}

\newcommand{\disp}{\displaystyle}
\DeclareMathOperator{\dist}{dist} \DeclareMathOperator{\diam}{diam}
\DeclareMathOperator{\li}{liminf} \DeclareMathOperator{\osc}{osc}
\DeclareMathOperator{\ess}{ess sup}
\DeclareMathOperator{\sign}{sign} \DeclareMathOperator{\sgn}{sgn}
\DeclareMathOperator{\e}{e} \DeclareMathOperator{\Lip}{Lip}
\DeclareMathOperator{\supp}{supp} \DeclareMathOperator{\di}{div}
\DeclareMathOperator{\bigO}{\Cal{O}}
\DeclareMathOperator{\Iff}{\Leftrightarrow}

\newcommand{\eps}{\varepsilon}
\newcommand{\vp}{\varphi}
\newcommand{\vte}{\vartheta}
\newcommand{\bdef}{\overset{\text{def}}{=}}

\newcommand{\al}{\alpha}
\newcommand{\be}{\beta}
\newcommand{\ga}{\gamma}
\newcommand{\de}{\delta}
\newcommand{\De}{\Delta}
\newcommand{\Ga}{\Gamma}
\newcommand{\te}{\theta}
\newcommand{\la}{\lambda}
\newcommand{\La}{\Lambda}
\newcommand{\om}{\omega}
\newcommand{\Om}{\Omega}
\newcommand{\si}{\sigma}

\newcommand{\ol}{\overline}
\newcommand{\np}{\newpage}
\newcommand{\nid}{\noindent}
\newcommand{\wh}{\widehat}

\newcommand{\iny}{\infty}
\newcommand{\del}{ \partial}
\newcommand{\su}{\subset}
\newcommand{\LP}{\Delta}
\newcommand{\gr}{\nabla}
\newcommand{\pri}{\prime}

\newcommand{\norm}[1]{\left\vert\left\vert #1\right\vert\right\vert}
\newcommand{\innp}[1]{\left< #1 \right>}
\newcommand{\abs}[1]{\left\vert#1\right\vert}
\newcommand{\set}[1]{\left\{#1\right\}}
\newcommand{\brac}[1]{\left[#1\right]}
\newcommand{\pr}[1]{\left( #1 \right) }
\newcommand{\pb}[1]{\left( #1 \right] }
\newcommand{\bp}[1]{\left[ #1 \right) }

\newcommand{\B}[1]{\ensuremath{\mathbf{#1}}}
\newcommand{\BB}[1]{\ensuremath{\mathbb{#1}}}
\newcommand{\Cal}[1]{\ensuremath{\mathcal{#1}}}
\newcommand{\Fr}[1]{\ensuremath{\mathfrak{#1}}}

\newcommand{\red}[1]{\textcolor{red}{#1}}

\newcommand{\N}{\ensuremath{\mathbb{N}}}
\newcommand{\Q}{\ensuremath{\mathbb{Q}}}
\newcommand{\R}{\ensuremath{\mathbb{R}}}
\newcommand{\Z}{\ensuremath{\mathbb{Z}}}
\newcommand{\C}{\ensuremath{\mathbb{C}}}

\newcommand\BD[1]{{\color{blue}#1}}
\newcommand\BDC[1]{{\color{cyan}#1}}
\newcommand\RD[1]{{\color{red}#1}}

\begin{document}

\title[Spectral inequalities in the plane]
{Spectral inequalities for Schr\"odinger equations and quantitative propagation of smallness  in the plane }
\author{Eugenia Malinnikova}
\address{
Department of Mathematics\\
Stanford University\\
Stanford, CA 94305, USA
}
\email{eugeniam@stanford.edu}
\author{Jiuyi Zhu}
\address{
Department of Mathematics\\
Louisiana State University\\
Baton Rouge, LA 70803, USA}
\email {zhu@math.lsu.edu }

\begin{abstract} This paper  deals with spectral inequalities for one-dimensional Schrödinger operators with potentials bounded between two increasing functions (weights). The spectral inequality allows one to estimate the norm of a function with bounded spectrum by its values on a certain sensor set. We say that a measurable subset of the real line is thick if the measure of the intersection of this set with any interval of fixed length is bounded from below. First, we consider thick sensor sets a large class of pairs of weights. For potentials constrained between two polynomials, spectral inequalities for a broad class of so-called generalized thick sets are analyzed. A quantitative dependence of the constants in the spectral inequalities on the density of the sensor sets, the growth rate of the potentials, and the spectral interval is established. The proofs rely on a new quantitative propagation of smallness (or quantitative Cauchy uniqueness) for elliptic equations in the plane.
\end{abstract}

\maketitle
\section{Introduction}
Let  $H=-\Delta +V(x)$ be the Schr\"odinger operator in  $ \R^d$ with potential $V(x)\in L^\infty_{loc}(\R^d)$.  It is known that if $\lim_{|x|\to \infty} V(x)=+\infty$, then the spectrum of $H$ is discrete, bounded from below,  and  infinite.  We denote the eigenvalues of $H$  as $\lambda_0\le \lambda_1\le...$, accounting for their multiplicities. Then $\lim_{k\to \infty} \lambda_k=\infty$.
 Let  $\phi_k\in L^2(\R^d)$  be the (normalized) eigenfunction 
corresponding to the eigenvalue  $\lambda_k$,
\begin{align}
-\Delta \phi_k +V(x)\phi_k=\lambda_k \phi_k \quad \text{ in } \mathbb{R}^d,
\label{eigen-k}
\end{align}
such that the eigenfunctions $\{\phi_k\}$ form an orthonormal basis of $L^2(\R^d)$.
Then $Ran(P_\lambda(H))$ denotes the span of eigenfunctions $\phi_k$ with $\lambda_k\leq \lambda$, and for any  $\phi\in Ran(P_\lambda(H))$, we have
\begin{align*}
\phi=\sum_{\lambda_k\leq \lambda} e_k \phi_k, \quad \mbox{with}\ \ e_k=\langle \phi_k, \phi\rangle.
\end{align*}
 
 The classical  uncertainty principle by Logvinenko and Sereda says that if the Fourier transform of a function $f$ is contained in an interval $[-R,R]$ and $\omega$ is a subset of the real line such that 
 \begin{align}
   |\omega\cap I|\ge \delta |I|,  \quad \mbox{ for any interval $I$ of length $a$},
   \label{thick-2}%
 \end{align}
  then 
\begin{align} \|f\|_{L^2(\R)}\le C(R,a,\delta)\|f\|_{L^2(\omega)}. 
\label{LS}%
\end{align} We call  measurable sets $\omega$ satisfying \eqref{thick-2} thick. The sharp dependence of the parameters  in the inequality (\ref{LS}) was obtained by Kovrijkine in \cite{K01}, who showed that 
$C(R,a,\delta)\le c_0|\delta|^{-c_1aR},$ where $c_0, c_1>0$ are absolute constants. 
In this article, we study similar estimates for functions in $Ran(P_\lambda(H))$, for a class of one-dimensional Schr\"odinger operators $H$. Such estimates are referred to as spectral inequalities. There is a vast literature on the subject; we refer the readers to the recent articles  \cite{BPS18, BM21, BJPS21, SSY23, ZZ23, MPS23, DSV23, DSV24, Zhu24, Wang24}, earlier literature \cite{LR95, JL99}, and the references therein.

The spectral inequality for $Ran(P_\lambda(H))$ has the following form
\begin{align}
 \|\phi\|_{L^2(\mathbb R^d)}\leq e^{C\lambda^\kappa}  \|\phi\|_{L^2(\omega)} \quad \text{for any } \phi\in  Ran(P_\lambda(H)),
 \label{spec-in}%
\end{align}
where  $\kappa>0,$ $C>0$, and  $\omega \subset \mathbb R^d$ is a measurable set. We are specially interested in the situations when $\kappa<1$, as this case is associated with null controllability, which we explain at the end of the introduction.

\if false
Similar spectral inequalities for linear combinations of the Laplace eigenfunctions on a compact Riemannian manifold $(\mathcal{M},g)$ are well studied. Let $\{\varphi_k\}$ be an orthonormal system of eigenfunctions on $\mathcal{M}$,
\begin{equation*}
-\Delta_g \varphi_{k} = \lambda_k\varphi_{k} \quad \mbox{on} \ \mathcal{M}.
\end{equation*}
Let $\omega$ be an open subset of $\mathcal{M}$ and let $\varphi= \sum_{\lambda_k< \lambda} e_k \varphi_k$,  it follows that
\begin{align}
 \|\varphi\|_{L^2(\mathcal{M})}\leq C_0 e^{C_1 \lambda^\frac{1}{2}}  \|\varphi\|_{L^2(\omega)},
 \label{spec-in-1}%
\end{align}
where $C_0$ and $C_1$ depend on $\mathcal{M}$ and $\omega$.
For individual eigenfunction, this inequality was proved by Donnelly and Fefferman in \cite{DF}. The sharp spectral inequality (\ref{spec-in-1}) for linear combinations of eigenfunctions
was shown in \cite{LR95}, \cite{JL99}.
The spectral inequality  (\ref{spec-in-1})  was used to study the null-controllability problem for the corresponding heat equation in \cite{LR95},  the null-controllability of thermo-elasticity system in \cite{LZ98}, and nodal sets of sum of eigenfunctions in \cite{JL99}. We refer the interested readers to the monograph \cite{LLR22a} for more extensive literature on the 
 spectral inequality (\ref{spec-in-1}). We also mention a recent article \cite{BM22} where more general measurable sets $\omega$ were considered.
\fi
The main results of this article are new  spectral inequalities (\ref{spec-in}) for the Schr\"odinger operators $H=-\partial^2_{x}+V(x)$ on $\mathbb R.$ We assume  that the potential is real-valued,  grows at infinity and is bounded from above and below by strictly increasing functions, $\Phi,\Psi:[0,+\infty)\to(0,+\infty)$, 
\begin{equation}\label{eq-phi-psi}%
\Phi(|x|)\le V(x)\le \Psi(|x|),
\end{equation}
where $\lim_{x\to\infty}\Phi(x)=\lim_{x\to\infty}\Psi(x)=+\infty$ and $\Phi(0)>0$. We note that the spectrum is discrete and $\lambda_1\ge \Phi(0)$.\
A particular case  is given by the following power growth condition for the potential 
 $V(x)$:  
\begin{equation}\label{V.growth}%
	c_1(|x| +1)^{\beta_1} \le V(x)\le c_2(|x| + 1)^{\beta_2}\quad {\text{for all}}\quad x\in \R,
\end{equation}
where  $c_1, c_2$, $\beta_1$, and $\beta_2$ are positive constants.

Our first result gives a sufficient condition for the spectral inequality \eqref{spec-in} to hold for any thick set $\omega$. 

  \begin{theorem}\label{th-mn1} Suppose that $H=-\partial^2_x+V(x)$, where $V$ satisfies \eqref{eq-phi-psi}, $\Phi$ has sub-exponential growth, i.e., for any $\beta>0$, there is $C_\beta$ such that $\Phi(|x|)\le C_\beta e^{\beta |x|}$, and 
  \begin{align}\Psi(\Phi^{-1}(\lambda)(1+o(1))=O(\lambda^{2\kappa}),\quad{\text{as}}\  \lambda\to\infty.
  \label{assum-2}
  \end{align} Then the spectral inequality \eqref{spec-in} holds for any thick set $\omega$. Moreover, one can choose  $C=C_0 a|\log \delta|$, where $a$ and $\delta$ are as in (\ref{thick-2}). 
 \end{theorem}

 We provide a list of examples of pairs of weights $\Phi$ and $\Psi$ that satisfy the conditions of this theorem in Section \ref{s-m1}. We believe that the condition on the interplay between $\Phi$ and $\Psi$ in Theorem \ref{th-mn1} is sharp. 
 A particular class of potentials covered by Theorem \ref{th-mn1} is given by \eqref{V.growth}, with $\kappa=\beta_2/(2\beta_1)$. 
 
 Let $\rho:\mathbb R\to \mathbb R_+$ be a real-valued function and $\tau\ge 0$. 
 We say that a measurable set $\omega\subset\R$ is $(\rho,\tau)$-thick if there exist constants $\gamma\in (0, \frac{1}{2})$ and $D>0$  such that
\begin{align}
	|\omega \cap I_{D\rho(x)}(x)|\geq {\gamma}^{{\langle x\rangle}^{\tau}}| I_{D\rho(x)}(x)|
	\label{set-def}%
\end{align}
for all $x\in \mathbb R$, where $I_r(x)=[x-r, x+r]$, $\langle x\rangle=(1+|x|^2)^{1/2}$, and $|\mathcal{S}|$ is the Lebesgue measure of the set $\mathcal{S}$. In particular, we consider $\rho(x)=\rho_s(x)={\langle x\rangle}^{s}$ in (\ref{set-def}).  We will refer to $(\rho_s,\tau)$-thick sets as to $(s,\tau)$-thick sets.
Our second result is the spectral inequality for the operator $-\partial_x^2+V(x)$ in dimension one with the potential $V$ of polynomial growth. 
\begin{theorem}
	Let $H=-\partial_x^2 +V(x)$. Assume  that $V$ satisfies  \eqref{V.growth} and $\omega$ is a $(s,\tau)$-thick set. \\
	(i) If $\frac{\beta_1-\beta_2}{2}\leq s<1$,  then
	\begin{align}
		\|\phi\|_{L^2(\mathbb R)}\leq e^{C\lambda^{\frac{\tau}{\beta_1}+\frac{s}{\beta_1}+\frac{\beta_2}{2\beta_1}}} \|\phi\|_{L^2(\omega)} \quad \text{for all } \phi\in Ran(P_\lambda(H)).
		\label{aim-res}
	\end{align}
	(ii) The inequality \eqref{aim-res} holds also for $s=1$, $\tau=0$.\\
	(iii)  If $s<\frac{\beta_1-\beta_2}{2}$, then
	\begin{align*}
		\|\phi\|_{L^2(\mathbb R)}\leq e^{C\lambda^{\frac{\tau}{\beta_1}+\frac{1}{2}}} \|\phi\|_{L^2(\omega)} \quad \text{for all } \phi\in Ran(P_\lambda(H)).
	\end{align*}
For these cases we can choose $C=C_0|\log \gamma|$, where $C_0$ depends only on $\beta_1, \beta_2, c_1, c_2, s$, and $D$  in \eqref{V.growth}, \eqref{set-def}.
	\label{th1}
\end{theorem}
We remark that  a $(1,\tau)$-thick set is any set of positive measure if $\tau>0$ (see Subsection \ref{ss:s=1}). In this case,  the inequality \eqref{aim-res} can be improved, as stated in the next result.
\begin{theorem}
	Let $H=-\partial_x^2 +V(x)$, where $V$ satisfies \eqref{V.growth},  and let $|\omega|>0$. Then there exists a constant $C$ depending only on $\beta_1, \beta_2, c_1, c_2, \omega$ such that 
	\begin{align*}
		\|\phi\|_{L^2(\mathbb R)}\leq e^{C\lambda^{\frac{1}{\beta_1}+\frac{\beta_2}{2\beta_1}}\log (\lambda+1)} \|\phi\|_{L^2(\omega)} \quad \text{for all } \phi\in Ran(P_\lambda(H)).
	\end{align*}
	\label{th2}
\end{theorem}
The main results are based on a new lemma providing a sharp quantitative Cauchy uniqueness estimate for a specific class of two-dimensional Schrödinger equations. This lemma establishes an explicit dependence of the constants on the supremum norm of the potential, without imposing any regularity requirements. The two-dimensional equations arise naturally when analyzing linear combinations of eigenfunctions of one-dimensional Schrödinger equations. The proof relies on the arguments of quasi-conformal mapping  and a representation of solutions for elliptic equations in the plane from \cite{KSW15}.
While general quantitative Cauchy uniqueness in any dimensions is studied in \cite{ARRV}, it does not yield the optimal dependence on the supremum norm of the potential that we want. We do not know whether our result extends to higher dimensions.

The spectral inequalities (\ref{spec-in}) in $\R^d$ have recently attracted significant attention. Let us review some of the literature.
For the harmonic oscillator $H=-\Delta +|x|^2$ in $\R^d$, the spectral inequality 
\begin{align}
	\|\phi\|_{L^2(\mathbb R^n)}\leq  e^{C\lambda^\frac{1}{2}}  \|\phi\|_{L^2(\omega)} 
	\label{spec-in-23}%
\end{align}
have been studied in, e.g. \cite{BPS18, BJPS21, MPS23, DSV23}
for open  or measurable sets $\omega$ with positive measure.  The arguments rely on real analyticity of the solution and the explicit formulas for the eigenfunctions. 
It is well known that the constant $e^{C\lambda^\frac{1}{2}}$ is sharp for spectral inequalities (\ref{spec-in-23}) for the harmonic oscillator. See also  \cite{M22} for spectral inequalities for $H=-\Delta +|x|^{2k}$ with integer $k$.



If $\omega$ is the union of  well-distributed open balls (e.g. $\omega=\bigcup_{j\in \mathbb Z^d} \mathbb{B}_{2^{-(1+|j|^\tau)}}(j))$,   the spectral inequalities (\ref{spec-in}) for  the Schr\"odinger operator have been shown in \cite{DSV24} and further sharpen in \cite{ZZ24} under the condition \eqref{V.growth} using Carleman estimates. See also \cite{M08} for the spectral inequalities in the case that  $\omega$ is an open cone.
For the measurable sets $\omega$ with positive measure in bounded domains or $\mathbb R^d$, the spectral inequalities for elliptic equations without potentials have been studied in e.g. \cite{BM21, BM22}. For the measurable sets $\omega$ with positive measure in $\mathbb R^d$ and local Lipschitz continuous $V(x)$ with growth in (\ref{V.growth}),  the spectral inequalities (\ref{spec-in}) were studied in \cite{Zhu24}. A crucial estimate for these spectral inequalities on the sets of positive measure relies on the propagation of smallness for the gradients in \cite{LM18}. In dimension two stronger results on the propagation of smallness of the gradients were obtained recently in \cite{zhu25, F23}.

In a recent paper \cite{Wang24}, the author dealt with some special range of $s$ in Case 3 in Theorem \ref{th1}, (e.g. $s<\frac{-\beta_2}{2}$) for the one-dimensional Sch\"odinger operator with the potential $V(x)$ satisfying \eqref{V.growth} and obtained the following sub-optimal spectral inequality 
\begin{equation}
\|\phi\|_{L^2(\mathbb R)}\leq e^{C\lambda^{{2\tau}/{\beta_1}+\frac{1}{2}}} \|\phi\|_{L^2(\omega)}.
	\label{sub-op}%
\end{equation}
Our approach shares some similarities with \cite{Wang24}, which quantifies  propagation of smallness results in the plane in \cite{SSY23}.  However, we obtain a sharp norm of quantitative Cauchy uniqueness result through a new strategy, which improves (\ref{sub-op}). Our proof also allows the study of a large range of $s$ in Theorem \ref{th1}.

\if false
The novelty of this paper lies in the following contributions. We reduce the study of the density of measurable sets of positive measure in (\ref{set-def}) to the density of these sensor sets in appropriate intervals. By combining sharp quantitative Cauchy uniqueness results for holomorphic functions with a representation of solutions for elliptic equations in the plane from \cite{KSW15}, we are able to establish quantitative Cauchy uniqueness for solutions of elliptic equations. The spectral inequalities are derived from these quantitative Cauchy uniqueness results, the decay of eigenfunctions, and a careful investigation of the role of the parameter
$s$ in the thickness condition. In contrast to some previous literature that focused on only one type of thick sensor set, we consider both weakly thick and strongly thick sets simultaneously and obtain sharp estimates. 

\fi
Spectral inequalities play an important role in  control theory. Let us consider the heat equation 
\begin{equation}
	\left\{
	\begin{aligned}
		u_t+Hu &= f(t,x)\mathds{1}_\omega \quad  &\text{in }& \ (0, T)\times \mathbb R^d, \\
		u(0, \cdot) &= u_0    &\text{on }& \  \mathbb R^d,
	\end{aligned}
	\right.
	\label{linear-one}%
\end{equation}
where $H=-\Delta +V(x)$, $V(x)$ is a real-valued potential function, $f\in L^2((0, T)\times \mathbb R^d)$, and $\omega\subset \mathbb R^d$ is a given measurable set. The equation (\ref{linear-one}) is said to be null controllable from the set $\omega$ at time $T$, if for any initial data $u_0\in L^2$, there exists $f\in L^2((0, T)\times \mathbb R^d)$, supported on $[0, T]\times \omega$,  such that $u(T)=0$. 
It was shown by Lions in \cite{L88}, using the so-called Hilbert uniqueness method (see also \cite{Cor07}), that the null controllability of (\ref{linear-one}) is equivalent to the observability estimate for the solution of the corresponding parabolic equation. By the  Lebeau-Robbiano approach \cite{LR95}, see also \cite[Chapter 7]{LLR22a} and references  therein,  the spectral inequality (\ref{spec-in}) implies the observability estimates. Specifically, it was demonstrated in \cite{NTTV20} that the heat equation is null controllable if the spectral inequality (\ref{spec-in}) holds with $0<\kappa<1$. Thus the results of Theorems \ref{th1} and \ref{th2} imply that the corresponding heat equations in (\ref{linear-one}) are null controllable if \begin{itemize}
\item[(i)] $\frac{\beta_1-\beta_2}{2}\le s<1$ and $\frac{\tau}{\beta_1}+\frac{s}{\beta_1}+\frac{\beta_2}{2\beta_1}<1$,
\item[(ii)] $s=1$, $\tau=0$, and $\frac{1}{\beta_1}+\frac{\beta_2}{2\beta_1}<1$, 
\item[(iii)] $s<\frac{\beta_1-\beta_2}{2}$ and $\frac{\tau}{\beta_1}<\frac{1}{2}$,
\item[(iv)] $|\omega|>0$ and $\frac{1}{\beta_1}+\frac{\beta_2}{2\beta_1}<1$.
\end{itemize}

The organization of the paper is as follows. 
A new quantitative Cauchy uniqueness estimate for a class of solutions of Schr\"odinger equations on the plane is proven in Section \ref{s:Cauchy}. This estimate is applied in Section \ref{s:two} to prove Theorem \ref{th-mn1}. 
In Section \ref{s:thick}, we  reformulate the generalized thickness condition by introducing appropriate sequences of intervals. Section \ref{s:two2} contains the proofs of Theorems \ref{th1} and
\ref{th2}. The Appendix includes the proof of propagation of smallness result with the explicit dependence on parameters and a technical estimate on the concentration of linear combinations of eigenfunctions. The letters $C$, $\hat{C}$, ${C}^{i}$, ${c}_i$, $C_i$ denote positive constants that do not depend on $\lambda$ or $u$, and may vary from line to line.

\section*{Acknowledgements}
Eugenia Malinnikova is partially supported by NSF grant DMS-2247185 and Jiuyi Zhu is partially supported by  NSF grant DMS-2154506.

\section{Propagation of smallness lemma}\label{s:Cauchy}
\subsection{Auxiliary estimate}
We  need the following results on the propagation of smallness for holomorphic functions. 
 \begin{lemma}
  Let $\mathcal{E}$ be a measurable set of $(-1, 1)$  of positive measure. Then, for any holomorphic function $h(z)$ in $\mathbb B_4$, 
\begin{align}
  \sup_{\mathbb B_2}|h|\leq  C \|h\|_{L^2(\mathcal{E})}^{\alpha} \sup_{\mathbb B_4}|h|^{1-\alpha},
  \label{three-hol}
\end{align}
where $\alpha=\frac{1}{C+C\log\frac{1}{|\mathcal{E}|}}$ and $C>1$ is a universal constant.
\label{lemma-pro}
 \end{lemma}

  The statement of Lemma \ref{lemma-pro} is  well known in potential theory. Some versions of this estimate were rediscovered in connection with quantitative unique continuation; see  Theorem 2.1 in \cite{M04} and  Proposition 2.1 in \cite{zhu25}. We reduce the inequality \eqref{three-hol} to the one proved in \cite{M04} in  Appendix. The dependence of $\alpha$ on the measure of $E$ in this propagation of smallness estimate plays an important role in the later arguments concerning the exponents in spectral inequalities.  

 \subsection{Main Lemma} 
Next lemma is a version of the quantitative Cauchy uniqueness result. 
Let $u(x,y)$ be a solution of the equation 
\begin{align}
-\Delta u+\partial_x (W(x) u)+ V(x)u=0 \quad \mbox{in} \ (-5, 5)\times (-5, 5),
\label{www-new}
\end{align}
where $V\ge 0$,  and 
\begin{align}
    \|V\|_{L^\infty}\leq M,  \quad \|W\|_{L^\infty}\leq K
    \label{assum-1-1}
\end{align}
with $K, M\geq 1.$
We prove an  explicit quantitative Cauchy uniqueness estimate for $u$ and obtain the dependence of the constants in the estimate on the $L^\infty$ norm of the potential $V(x)$ and $W(x)$. Since $V(x)$ is assumed to be nonnegative, we combine the Lemma \ref{lemma-pro} and some ideas in the study of Landis' conjecture in plane in \cite{KSW15}. See also \cite{ZZ24} for applications of results as Lemma \ref{propation-1} to observability inequalities of heat equations.
\begin{lemma}
    Let $u(x,y)$ be a solution of (\ref{www-new}) and $\frac{\partial u}{\partial y}=0$ on $Q_5\cap\{y=0\}$. Assume that $V(x)\geq 0$ and (\ref{assum-1-1}) holds. Let $\mathcal{E}$ be a measurable set on the line $\mathbb B_{1}\cap\{y=0\}$ with  $|\mathcal{E}|>0$.Then
    \begin{align}
  \|u\|_{L^2(\mathbb B_2)}\leq e^{C(\sqrt{M}+K)}  \|u\|^{\alpha}_{L^2(\mathcal{E})}   \|u\|_{L^2(\mathbb B_4)}^{1-\alpha},
  \label{nono-1}
\end{align}
where $\alpha=\frac{1}{C+C\log\frac{1}{|\mathcal{E}|}}$.
\label{propation-1}
\end{lemma}
\begin{proof}
We construct a multiplier to build a second order elliptic equation without lower order terms.  We  first prove the existence of a solution $w(x)$ of the following ordinary differential equation
\begin{align}
\begin{cases}
     -w'' -W(x) w'+ V(x) w  =0  \quad \quad \mbox{in} \ L:=(-5, 5), \\
 w(-5)=w(5)  = e^{5(\sqrt{M}+K)} 
\end{cases}
\label{build}
\end{align}
Choose ${w}_1= e^{(\sqrt{M}+K)x}$. Then \begin{align*}
    -{w}''_1 -W(x)w_1'+V(x) {w}_1 
    \leq \pr{ -(\sqrt{M}+K)^2+ (\sqrt{M}+K)K+M} e^{(\sqrt{M}+K)x}
    \leq 0. 
\end{align*}
Thus, ${w}_1$ is a sub-solution of (\ref{build}). Let ${w}_2= e^{5(\sqrt{M}+K)}$. Then $$-{w}''_2 -W(x)w'_2 +V(x) {w}_2 \geq 0\quad \mbox{in} \ L$$ as $V(x) \geq 0$. Thus, ${w}_2$ is a super-solution of (\ref{build}). By the subsolution and supersolution method, see for example \cite{DH96}, there exists a solution $w$ of (\ref{build}) such that \begin{align*}{w}_1\leq w\leq {w}_2\quad \mbox{in} \ L .\end{align*} Therefore, we have 
\begin{align} e^{-5(\sqrt{M}+K) }\leq w(x)\leq e^{5(\sqrt{M}+K) }\quad \mbox{in} \ L.
\label{www-1}%
\end{align} 
By the elliptic estimates, it follows that 
\begin{align} \| w'\|_{L^\infty([-4,4])} \leq e^{C(\sqrt{M} +K)}
\label{www-2}\end{align} 
for some $C>1$.
We consider $u_1(x,y)=\frac{u(x, y)}{w(x)}$. It is easy to see that $u_1(x,y)$ satisfies 
\begin{align*}
    {\rm div}(w^2 (\nabla u_1-W_1(x) u_1))=0 \quad \mbox{in} \ (-5, 5)\times (-5, 5),
\end{align*}
where $W_1(x)=(W(x), 0)$.

We will use some technique from complex analysis, which has been applied in \cite{KSW15} in the study of Landis' conjecture in the plane. We construct a stream function $u_2$ associated with $u_1$ in $\mathbb B_5$. That is, $u_2$ satisfies 
\begin{align}\label{anti-sy}
\begin{cases}
     \ \ \partial_y u_2   = w^2 \partial_x u_1-w^2 Wu_1, \\
-\partial_x u_2   = w^2 \partial_y u_1.
\end{cases}
\end{align}
We can choose $u_2(0)=0$. Let $g=w^2 u_1+i u_2.$ 
Then
\begin{align*}
    \overline{\partial} g=q_0
    (g+\bar g)    \quad \mbox{in} \ \mathbb B_5, 
\end{align*}
where    $q_0(x+iy)= \frac{w'(x)}{2w(x)}+\frac{1}{4}W(x).$
Furthermore,
it can be reduced to 
$\overline{\partial} g=\tilde{q}_0 g$   
with
\begin{align*}
\tilde{q}_0(z)= \left\{ 
\begin{array}{lll}
{q}_0(z)+ \frac{{q}_0(z) \overline{g(z)} }{g(z)}, \quad \mbox{if} \ g\not=0, \nonumber \\
0,  \quad \mbox{otherwise}.
\end{array}
\right.
\end{align*}
 The following estimates 
\begin{align}
\|\tilde{q}_0\|_{L^\infty(\mathbb B_{9/2})}\leq C(\sqrt{M}+K)
  \label{qqq-1}
\end{align}
 was shown in Lemma 3.1 in \cite{KSW15}.
Following \cite{KSW15}, we also define
\begin{align*}
    Q(z)=-\frac{1}{\pi}\int_{\mathbb B_{\frac{9}{2}}}\frac{\tilde{q}_0(\xi)}{\xi-z} d\xi,
\end{align*}
such that $\overline{\partial}Q=\tilde{q}_0$ in $\mathbb{B}_{\frac{9}{2}}$, then
$h(z)=e^{-Q(z)}g(z)$
is a holomorphic function in $\mathbb B_{\frac{9}{2}}$.
Moreover, the bound on $\tilde{q}_0$ in (\ref{qqq-1}) yields that 
\begin{align} 
\|Q\|_{L^\infty (\mathbb B_{9/2})}\leq C(\sqrt{M}+K).
\label{WWW-1}
\end{align}

Next we apply the propagation of smallness results of Lemma \ref{lemma-pro} to the holomorphic function $h(z)=e^{-Q(z)}g(z)$. We have
\begin{align*}
  \sup_{\mathbb B_2}|e^{-Q(z)}g(z)|\leq C \|e^{-Q(z)}g(z)\|^{\alpha}_{L^2(\mathcal{E})} \sup_{\mathbb B_{4}}|e^{-Q(z)}g(z)|^{1-\alpha}.   
\end{align*}
Recall that $g=w^2u_1+iu_2$. Taking into account (\ref{www-1}) and (\ref{WWW-1}), we derive
\begin{align}
  \sup_{\mathbb B_2}|u_1|\leq e^{C(\sqrt{M}+K)} \|w^2 u_1+iu_2\|^{\alpha}_{L^2(\mathcal{E})} \sup_{\mathbb B_{4}}|w^2 u_1+iu_2 |^{1-\alpha}.   
  \label{hada-2}
\end{align}
Since $\frac{\partial u}{\partial y}=0$ on $\{ (x,y)\in \mathbb B_4|y = 0\}$ and $w(x)$ depends only on $x$, we get $\frac{\partial u_1}{\partial y}=0$ {on $\{(x,y)\in \mathbb B_4| y = 0\}$}.  Furthermore, (\ref{anti-sy}) implies that  $\frac{\partial u_2}{\partial x}=0$ on $\{ (x,y)\in \mathbb B_4|y = 0\}$. As $u_2(0)=0$, then $u_2=0$ on $\{(x,y)\in \mathbb B_1|y=0\}$. Thus, $u_2=0$ on $\mathcal{E}$. Finally, we estimate the $L^\infty$ norm of $u_2$ in (\ref{hada-2}).
We have
\begin{align*}\nabla u_1=\nabla \left(\frac{u}{w}\right)=\frac{\nabla u w-\nabla w u}{w^2}.
\end{align*}
Then 
\begin{align}
  \sup_{\mathbb B_{4}}|u_2 |   \leq C\sup_{\mathbb B_{4}} |\nabla u_2| 
  \leq C  e^{C(\sqrt{M}+K)} (\sup_{\mathbb B_{4}} |\nabla u_1| +\sup_{\mathbb B_{4}} |u_1|) 
\leq e^{C_1(\sqrt{M}+K)} \sup_{\mathbb B_{\frac{9}{2}}} | u|,
  \label{hada-3}
\end{align}
where we used the estimates (\ref{www-1}),  (\ref{www-2}), the equations \eqref{anti-sy}, and elliptic estimates for $u$ which solves (\ref{www-new}).
Therefore, by (\ref{hada-2}) and (\ref{hada-3}) and the estimates for $w$, we obtain
\begin{align*}
  \sup_{\mathbb B_2}|u|\leq e^{C(\sqrt{M}+K)}  \|u\|^{\alpha}_{L^2(\mathcal{E})} \sup_{\mathbb B_{\frac{9}{2}}}|u |^{1-\alpha}.    
\end{align*}
By the elliptic estimates and rescaling, (\ref{nono-1}) follows.
\end{proof}

\section{Spectral inequalities: case of thick sets}\label{s:two}
\subsection{Localization of linear combinations of eigenfunctions and the lift construction}\label{s:main} 
It is well-known that as the potential $V(x)$ grows at infinity, the eigenfunctions in (\ref{eigen-k}) are well localized and decay exponentially. We need the
 following lemma, which quantifies the decay property of  linear combinations of eigenfunctions.  
\begin{lemma}\label{l:twow} Suppose that $H=-\partial_x^2+V$, where $\Phi(x)\le V(x)\le \Psi(x)$, where $\Phi,\Psi$ are increasing and tend to infinity when $|x|\to\infty$.  Let $\phi\in Ran(P_\lambda(H))$ with $\lambda>\Phi(0)$, and let $R_t=\Phi^{-1}(t)$ for $t\ge \Phi(0)$. Then for $r>C+\frac{1}{2}\log(\lambda+2)+\frac{1}{2}\log R_{\lambda+2}$, we have
\[\|\phi\|^2_{L^2(\R)}\le 2\|\phi\|^2_{L^2(-R_{\lambda+2}-r,R_{\lambda+2}+r)}.\]
\end{lemma}
We give a proof of this result in Appendix. For the case where $V$ satisfies
\begin{equation*}
	c_1(|x| +1)^{\beta_1} \le V(x)\le c_2(|x| + 1)^{\beta_2},
\end{equation*}
the above inequality takes the form \eqref{L2} below, see also \cite{DSV24, GY12, ZZ23} for similar inequalities.
\begin{corollary}
There exists a constant $\hat{C}$, depending on $\beta_1$, $c_1$, $\beta_2$ and $c_2$ such that for 
 $\phi\in Ran(P_\lambda(H))$, we have
\begin{align}
 \|\phi\|^2_{L^2(\mathbb R)} 
\leq 2\|\phi\|^2_{L^2(-\hat{C}\lambda^{1/\beta_1}, \hat{C}\lambda^{1/\beta_1})}.
\label{L2}
\end{align}
\label{propo-new}
\end{corollary}

We consider  linear combinations of eigenfunctions and use the lifting argument (also known as the ghost dimension argument) to construct  solutions to elliptic equations in the plane. We recall that $\lambda_1>0$ as $V(x)\geq 0$. 
Let $\phi=\sum_{0<\lambda_k\leq \lambda} e_k \phi_k$ for $e_k=(\phi_k, \phi)$. We define 
\begin{align} u(x, y)=\sum_{0<\lambda_k\leq \lambda} e_k \cosh(\sqrt{\lambda_k} y)\phi_k(x),
\label{uuu1}\end{align}
such that $u(x,0)=\phi(x)$ and  $\frac{\partial u}{\partial y}(x,0)=0$. It is easy to see that
$-\Delta u+V(x)u=0$ in $\mathbb R^2.$ 

\subsection{Proof of Theorem \ref{th-mn1}}\label{ss:twow} 
Let $u$ be define by \eqref{uuu1}. We note that  $u(\cdot,y)\in Ran(P_\lambda(H))$ for each $y$. Then, applying Lemma \ref{l:twow}, we obtain
\begin{align}
\|u(\cdot, y)\|^2_{L^2(\mathbb R)}\leq 2 \|u(\cdot,y)\|^2_{L^2(-R_{\lambda+2}-r,R_{\lambda+2}+r)},
\label{new-twow}
\end{align}
where $R_{\lambda+2}$ and $r$ are as in Lemma \ref{l:twow}. 

Let $\omega$ be a thick set, so \eqref{thick-2} is satisfied.
We partition $\R$ into intervals $I_n=[an,a(n+1)]$, where $n\in\Z$ and \eqref{thick-2} implies that $|\omega\cap I_n|\ge \delta|I_n|$ for each $n$. We fix $n$ and define
$\tilde{u}_n(x,y)=u(a(n+1/2)+ax,ay)$. Then $-\Delta \tilde{u}_n+a^2\tilde{V}_n\tilde{u}_n=0$ in $(-5,-5)\times(-5,5)$, where $\tilde{V}_n(x,y)=V(a(n+1/2)+ax)$. 
Let $Q_n=I_n\times[-a/2,a/2]\subset \mathbb B_{a}(X_n)$, where $X_n=(a(n+1/2),0)$. Then $\mathbb B_{2a}(X_n)\subset 4Q_n=4I_n\times[-2a,2a]$. Let $\tilde{\omega}_n=|I_n|^{-1} (\omega\cap I_n-X_n)$.
Applying Lemma \ref{propation-1} with $W(x)=0$, we conclude that
\begin{align*}
 \|\tilde{u}_n\|_{L^2(\mathbb B_1)} \leq 
 e^{Ca \sqrt{M_n} }\|\tilde{u}_n\|_{L^2(\tilde{\omega}_n)} \|\tilde{u}_n\|_{L^2(\mathbb B_{2})},
\end{align*}
where $\alpha=\frac{1}{C+C\ln \frac{1}{|\tilde{\omega}|}}$  and $M_n=\max_{[-3,3]} \tilde{V}_n$.
Rescaling back to $u$, we have
\begin{align*}
\|u\|_{L^2(Q_n)}\le ae^{Ca\sqrt{M_n}}\|u\|_{L^2(\omega\cap I_n)}^{\alpha}\|u\|_{L^2(4Q_n)}^{(1-\alpha)},    
\end{align*}
where $M_n=\max_{4I_n}V$ and $\alpha=\frac{1}{C(1+\ln 1/\delta)}$. We apply the Young's inequality, $ab\le \frac{a^p}{p}+\frac{b^q}{q}$ with $p=\frac{1}{\alpha}$ and $q=\frac{1}{1-\alpha}$, to the last product, 
\begin{equation}\label{eq:Hnew}
\| u\|^2_{L^2(Q_n)}\le e^{Ca\sqrt{M_n}} \left(K^{1/\alpha}\alpha \|u\|^2_{L^2(\omega\cap I_n)}+ K^{-1/(1-\alpha)}(1-\alpha)\|u\|^2_{L^2(4Q_n)}\right)
\end{equation}
for any $K>0$; we specify the choice of $K$ later.

 Choose  $N\in\mathbb{Z}_+$ such that $aN<R_{\lambda+2}+r\le a(N+1)$. Then 
 \[(-R_{\lambda+2}-r, R_{\lambda+2}+r)\subset\bigcup_{-N-1\le n\le N} I_n.\] Recall also that  $u(x,y)$ is defined by (\ref{uuu1}). We have
 \begin{align*}
     \sum_{n=-(N+1)}^N\|u\|^2_{L^2(Q_n)} &=
     \int_{-a/2}^{a/2}\int_{-a(N+1)}^{a(N+1)}|u(x,y)|^2dxdy\\ &\ge \frac{1}{2}\int_{-a/2}^{a/2}\int_{\R}|u(x,y)|^2dxdy \\ &=\frac{1}{2}\int_{-a/2}^{a/2}\sum_{\lambda_k\leq \lambda}e_k^2\cosh^2\sqrt{\lambda_k}ydy\\ & \ge \frac{a}{2}\sum_{\lambda_k\leq \lambda} e_k^2
     =\frac{a}{2}\|\phi\|_{L^2(\R)}^2,\end{align*}
       where we used (\ref{new-twow}) and the fact that $\phi_k$ are orthonormal.
     Similarly, using that each point is covered by at most $4$ of the intervals $4I_n$,  we have
     \begin{align*}
     \sum_{n=-(N+1)}^N\|u\|^2_{L^2(4Q_n)}&\le 4\int_{-2a}^{2a}\int_{\R}|u(x,y)|^2dxdy \\ &=4\int_{-2a}^{2a}\sum_{\lambda_k\leq \lambda}e_k^2\cosh^2\sqrt{\lambda_k}ydy\\ &\le 16a e^{4\sqrt{\lambda}a}\sum_{\lambda_k\leq  \lambda} e_k^2=16ae^{4\sqrt{\lambda}a}\|\phi\|^2_{L^2(\R)}.
     \end{align*} 
   
 Let $M_N= \max_{(-a(N+2),a(N+2))}V$. We sum the inequalities \eqref{eq:Hnew} for $n$ from $-(N+1)$ to $N$ and use the above estimates to obtain
 \[\|\phi\|_{L^2(\R)}\le e^{Ca\sqrt{M_N}}\left(CK^{1/\alpha}\|\phi\|_{L^2(\omega)}+CK^{-1/(1-\alpha)}e^{2\sqrt{\lambda}a}\|\phi\|_{L^2(\R)}\right).\]
 We choose $K=C_1 e^{(Ca\sqrt{M_N}+2a\sqrt{\lambda})(1-\alpha)}$ such that  $CK^{-1/(1-\alpha)} e^{Ca\sqrt{M_N}}e^{2\sqrt{\lambda}a}=\frac{1}{2}$.  Then we can incorporate the second term in the last inequality into the left hand side.
 Hence we obtain
 \begin{align}\|\phi\|_{L^2(\R)}\le Ce^{Ca\frac{(1-\alpha)}{\alpha}(\sqrt{M_N}+\sqrt{\lambda})}\|\phi\|_{L^2(\omega)}.\label{assum-1}
 \end{align}
 
We need to estimate $M_N$. From the assumptions in Lemma \ref{l:twow}, we have
\[a(N+2)\le R_{\lambda+2}+r+2a\leq\Phi^{-1}(\lambda+2)+\frac{1}{2}\log(\lambda+2)+\frac{1}{2}\log \Phi^{-1}(\lambda+2)+C.\] 
We want to show that $\log(\lambda+2)=o(\Phi^{-1}(\lambda+2))$, which follows from the fact that $\Phi(|x|)\le C_\beta e^{\beta |x|}$ for any $\beta>0$. It is clear that the other two terms are also $o(\Phi^{-1}(\lambda+2))$.  
By the assumption (\ref{assum-2}), we have 
\[M_N\le \Psi(a(N+2))\le \Psi(\Phi^{-1}(\lambda+2)(1+o(1)))=O((\lambda+2)^{2\kappa})=O(\lambda^{2\kappa}),\quad \lambda\to\infty.\]
Note that $\Psi(\Phi^{-1}(\lambda))\geq \lambda $ implies $\kappa\geq \frac{1}{2}$. 
Hence, recalling that $\alpha=(C+C\ln \frac{1}{\delta})^{-1}$, the estimate (\ref{assum-1}) implies that 
\[\|\phi\|_{L^2(\R)}\le e^{C\lambda^\kappa}\|\phi\|_{L^2(\omega)},\]
where $C=C_0 a|\ln \delta|$. This completes the proof of Theorem \ref{th-mn1}.

\subsection{Examples of admissible potentials}\label{s-m1} Our main examples are  potentials of polynomial growth that satisfy \eqref{V.growth}.  
However, we may also allow $V$ to grow either faster or slower than polynomially.
\begin{example} Suppose that $\Phi(|x|)=c_1\exp (a_1|x|^\gamma)$, where $\gamma<1$. Then we can take $\Psi(|x|)=c_2\exp (a_2|x|^\gamma)$, and $\kappa>a_2/2a_1\geq 1/2$. Clearly, $\Phi(|x|)\le C_\beta e^{\beta |x|}$ for any $\beta>0$ and $|x|\geq c$ for some constant $c$. Moreover, if $\Phi(|x|)=\lambda$, then
\[\Psi((1+\varepsilon)|x|)=c_2\exp(a_2(1+\varepsilon)^\gamma |x|^\gamma)\le C\lambda^{(1+\varepsilon)^\gamma a_2/a_1}\le \lambda^{2\kappa},\]
when $a_2<2\kappa a_1$ and $\varepsilon$ is small enough.\end{example}

\begin{example} Suppose that \[\Phi(x)=c_1\exp(d_1(\log^\delta(|x|+1)))\quad {\text{and}}\quad \Psi(x)=c_2\exp(d_2(\log^\delta(|x|+1))),\]  where $\delta>0$.  Then the assumptions of  Theorem \ref{th-mn1} are satisfied with $\kappa=d_2/2d_1\geq 1/2$.
\end{example}

\begin{example}
    Another pair of admissible weights is given by
    \[\Phi(x)=c_1\log^{\tau_1}(|x|+1)\quad {\text{and}}\quad \Psi(x)=c_2\log^{\tau_2}(|x|+1),\] with $\kappa=\tau_2/2\tau_1\geq 1/2$.
\end{example}

\section{Generalized thick sets}\label{s:thick}
\subsection{Definition}
We recall that  a measurable set $\omega\subset\R$ is $(s,\tau)$-thick if there exist constants $\gamma\in (0, \frac{1}{2})$ and $D>0$  such that
\begin{align}
	|\omega \cap I_{D\rho(x)}(x)|\geq {\gamma}^{{\langle x\rangle}^{\tau}}| I_{D\rho(x)}(x)|
	\label{set-def1}
\end{align}
for all $x\in \mathbb R$, where $I_r(x)=[x-r, x+r]$, $\langle x\rangle=(1+|x|^2)^{1/2}$, and $\rho(x)=\rho_s(x)={\langle x\rangle}^{s}$.  

\begin{remark}
\begin{enumerate}
	\item If $s=0$ and $\tau=0$, the assumption (\ref{set-def1}) reduces to 
	$		|\omega \cap I_{D}(x)|\geq {\gamma}| I_{D}(x)|,$
which is  the definition of a thick set in \eqref{thick-2}.

\item The case $s=-1$, $\tau=0$ was considered in \cite{SVW98} in a different context of the estimates for Fourier localization operators. 
	
	
	\item We can relax the definition (\ref{set-def1}) by requiring that 
	(\ref{set-def1}) holds only for $|x|$  large enough. Indeed, if (\ref{set-def1}) only holds for $|x|\geq M$ for some large $M$, then we can choose a new large $D$ and a small $\gamma$ such that $\omega$ satisfies 
	(\ref{set-def1}) for all $x\in \mathbb R$.
	\end{enumerate}
\end{remark}

The aim of this section is to reformulate the condition \eqref{set-def1}. Instead of working on intervals with arbitrary centers, we consider a partition of the real line into intervals and require estimates on the measure of the sensor set in each of these intervals.

\subsection{Case 1: $s<1$}  
For $n\ge 1$,  we define $x_n=an^{1/(1-s)}$, where $a>0$ will be chosen later.  Let $I_0=(-x_1, x_1)$, $I_n=(x_n, x_{n+1})$  when $n=1,2,...$,  and $I_{-n}=-I_n$. Moreover, we define $x_0=0$ and $x_{-n}=-x_n$.

\begin{lemma}\label{l:set-n} A set $\omega$ is $(s,\tau)$-thick if and only if there exist $a>0$ and $\gamma_1\in(0, 1/2)$ such that, for $x_n$ and $I_n$ defined above, 
	\begin{equation}|\omega\cap I_n|\ge \gamma_1^{\langle x_n\rangle^\tau}|I_n| \quad {\text{for all integer}}\ n.
		\label{set-n}
		\end{equation} 
\end{lemma}

\begin{proof} Suppose that $\omega$ is $(s,\tau)$-thick and \eqref{set-def1} holds with $\rho_s(x)=\langle x\rangle ^s$ and some $\gamma$, $D>0$, we will prove  (\ref{set-n}).  Let $y_n=(x_n+x_{n+1})/2$. 
Then
\[y_{n}=\frac{a}{2}\left(n^{1/(1-s)}+(n+1)^{1/(1-s)}\right)\le \frac{a}{2}\left(1+2^{1/(1-s)}\right)n^{1/(1-s)}\le a 2^{1/(1-s)}n^{1/(1-s)},\]
since $l_s=2^{1/(1-s)}>1$.
Thus,  $y_n\in (x_n, l_sx_n)$.  
		
		First, we consider $I_0=(-a,a)$. Assume that $a>D$. Applying \eqref{set-def} with $x=0$, we get
	\[|\omega\cap (-a,a)|\ge |\omega\cap (-D,D)|\ge 2\gamma D\ge 2\gamma_0 a,\]
	where $\gamma_0=D\gamma /a$.
	
	To obtain \eqref{set-n} for $n\neq 0$, we claim that we can choose $a>\max\{1, D\}$ such that the following inequality holds
	\begin{equation}
		D \langle y_n\rangle^s<\frac{x_{n+1}-x_n}{2}
		\label{choose-a}
		\end{equation}
		for all $n\ge 1$.
	For  $s<0$,  we have $D\langle y_n\rangle^s\le D(1+x_n^2)^{s/2}<Dx_n^{s}$, and for $s>0$,
	we have \[D\langle y_n\rangle^s\le D(1+(l_sx_n)^2)^{s/2}\le D_1x_n^s,\]
	where $D_1=D(1+l_s^2)^{s/2}$ and we have used the fact that $x_n\ge a>1$.
	Thus, 
    \begin{align}
    \label{last}
  D\langle y_n\rangle^s<D'a^sn^{s/(1-s)}\end{align}
	for some $D'$ which depends only on $D$ and $s$. 
	On the other hand, 
\begin{align}\frac{x_{n+1}-x_n}{2}= \frac{x_n}{2}\left(\left(1+\frac{1}{n}\right)^{1/(1-s)}-1\right)\ge c_sx_nn^{-1}=c_san^{s/(1-s)}, 
\label{last-1} \end{align}
	where we used the estimates $\frac{(1+\frac{1}{n})^{1/1-s}-1}{2}\geq c_s n^{-1}$
    with $c_s=\frac{2^{s/(1-s)}}{2(1-s)}$ for $s\le 0$ and $c_s=\frac{1}{2(1-s)}$ for $s>0$.  Now, since $s<1$,  we can choose $a$ large enough, depending on $D$ and $s$, such that $D'\le c_sa^{1-s}$. Then \eqref{choose-a} follows from (\ref{last}) and (\ref{last-1}). 
	
	By the choice of $y_n$, we have $x_n\leq y_n-D\langle y_n\rangle^s$ and $x_{n+1}\geq y_n+D\langle y_n\rangle^s$. Then we get
	\[|\omega\cap I_n|\ge |\omega\cap I_{D\rho_s(y_n)}(y_n)|\ge \gamma^{\langle y_n\rangle^\tau} |I_{D\rho_s(y_n)}(y_n)|,\]  where the definition of (\ref{set-def1}) is used.
	We also have \begin{align*}|I_{D\rho_s(y_n)}|=2D\langle y_n\rangle^s
    \ge 2D \min\{ \langle x_n\rangle^s, \langle l_s x_n\rangle^s\}
    \ge c_{a,s,D}|I_n| \end{align*} for some constant $c_{a,s,D}$, where we used (\ref{last-1}) as $|I_n|=|x_{n+1}-x_n|$.
    Hence
	\[|\omega\cap I_n|\ge c_{a,s,D}\gamma^{\langle y_n\rangle^\tau}|I_n|.\]
	We want to choose $\gamma_1$ such that 
	\[c_{a,s, D}\gamma^{\langle y_n\rangle^\tau}\ge \gamma_1^{\langle x_n\rangle^\tau}\]
	for every $n\ge 1$. We see that $y_n<x_{n+1}\le 2^{1/(1-s)}x_n$ and then $\langle y_n\rangle\le 2^{1/(1-s)} \langle  x_n \rangle $ holds.
    Thus, the required inequality (\ref{set-n}) holds for 
	$\gamma_1\le \gamma_*=\min\{1, c_{a,s, D}\}\gamma^{\beta}$ with $\beta=2^{\tau/(1-s)}$.	Then \eqref{set-n} holds for all $n\ge 1$. Similarly we see that it holds for $n=-1,-2,... .$  Finally, by defining $\gamma_1=\min\{\gamma_0, \gamma_*\}$, we show that (\ref{set-def1}) implies (\ref{set-n}).
	
	Now we want to show the opposite. That is, suppose that \eqref{set-n} holds fol all $n$, we claim that $\omega$ is $(s,\tau)$-thick. Let $x\in I_n$, where $n\ge 1$. Then
	\begin{align}
	    \max\{x_{n+1}-x, x-x_n\}\le x_{n+1}-x_n\le d_san^{s/(1-s)},
        \label{haha-100}
	\end{align}
 {where we used $(1+\frac{1}{n})^{1/1-s}-1\leq d_s n^{-1}$ with  $d_s=\frac{2^{s/(1-s)}}{(1-s)}$ for $s\geq 0$, and $d_s=\frac{1}{1-s}$ for $s<0$.} We can also show that $\rho_s(x)\geq C_{s} a^{s}n^{s/(1-s)}$ for $x\in I_n$ and some $C_s$ depending on $s$. Thus, 
  from (\ref{haha-100}),  we can choose $D$ depending on $a$ and $s$ such that $I_{D\rho_s(x)}(x)\supset I_n$. On the other hand, fixing $D$, by the lower bound given in (\ref{last-1}), there exists a small constant $b_{s,a,D}<1$ depending $a, s, D$ such that we  have $|I_n|\ge b_{s,a,D}|I_{D\rho_s(x)}|$ for any $x\in I_n$. It follows from \eqref{set-n} that
	\[|\omega\cap I_{D\rho_s(x)}(x)|\ge |\omega\cap I_n|\ge \gamma_1^{\langle x\rangle^\tau}|I_n|\ge b_{s,a,D}\gamma_1^{\langle x\rangle^\tau}|I_{D\rho_s(x)}|.\]
	Then we  can choose $\gamma^\ast>0$ small enough such that $b_{s,a,D}^{1/\langle x\rangle^\tau} \gamma_1\geq b_{s,a,D}\gamma_1\geq \gamma^\ast $. Hence
    \eqref{set-def1} holds with $\gamma=\gamma^\ast$ for $n\geq 1$.
	Finally, if $x\in(-a,a)$, we have $I_{D\rho_s(x)}\supset(-a,a)=I_0$ when $D>2a$. Then we can choose $\gamma_0=\frac{\gamma_1 a}{D(1+a^2)^{s/2}}$  such that
	\[|\omega\cap I_{D\rho_s(x)}|\ge|\omega\cap I_0|\ge 2\gamma_1a\ge 2\gamma_0 D\rho_s(x) \geq \gamma_0^{\langle x\rangle^\tau} |I_{D\rho_s(x)}(x)|.  \]
    Therefore, by letting $\gamma=\min (\gamma^\ast, \gamma_0)$, we show that  (\ref{set-n}) implies  (\ref{set-def1}). Hence the lemma is completed.
\end{proof}

\begin{remark}\label{r:overlap}  Note that  the intervals $\{AI_n\}$ have finite number of  overlaps for any given $A$. Indeed if $m\ge n$ and $AI_m\cap A I_n\neq\emptyset$, then 
\[|I_{n}|+...+|I_{m}|=|x_{m+1}-x_n|\le \frac{A+1}{2}(|I_n|+|I_m|).\]
On the other hand, we have $\lim_{n\to\infty}|I_{n+1}|/|I_n|=1$. Thus $|m-n|\le CA$. This shows that $AI_n$ may intersect only intervals $AI_m$ with $m\in(n-A, n+CA)$. Therefore for any $x\in\R$, $\#\{j: x\in I_j\}\le CA.$
\end{remark}

	\subsection{Case 2: $s=1$}\label{ss:s=1} We will consider the cases $\tau=0$ and $\tau>0$, respectively.
    
    \begin{lemma}\label{l:set-nn} (i) A set $\omega$ is $(1,0)$-thick  if and only if there is $\gamma_1$ such that
	\begin{equation}
		|\omega\cap(-n,n)|\ge \gamma_1 n
		\label{(1,0)-thick}
		\end{equation}
        for all $n\ge n_0$.\\
        (ii) If $\tau>0$, a set $\omega$ is $(1,\tau)$-thick if and only of it has a positive measure.
\end{lemma}
\begin{proof} (i) Suppose that \eqref{set-def1} holds with $\rho_1(x)=\langle x\rangle$, $\tau=0$, and some $\gamma$,  $D>0$. We can choose $x= n/(3(D+1))$ such that $I_{D\langle x\rangle}(x)\subset(-n,n)$, where  $n$ is large enough. Then
	\[|\omega\cap(-n,n)|\ge |\omega\cap I_{D\langle x\rangle}|\ge 2\gamma D\langle x\rangle\ge \gamma_1 n,\]
    where $\gamma_1=\frac{2\gamma D}{3(D+1)}$
On the other hand, suppose that \eqref{(1,0)-thick} holds for all $n\ge n_0$. We choose $D=\max\{2, 2n_0\}$. For $x<n_0$, we have
\[|\omega\cap I_{D\langle x\rangle}(x)|\ge |\omega\cap(x-D,x+D)|\ge |\omega\cap(-n_0, n_0)|\ge \gamma_1n_0\ge \gamma D\langle x\rangle,\]
for $\gamma=\frac{\gamma_1}{2D}$ small enough. Moreover, for $|x|\in [n, n+1)$ where $n\ge n_0$, we have
\[|\omega\cap I_{D\langle x\rangle}(x)|\ge |\omega\cap(-n,n)|\ge \gamma_1n\ge \gamma D\langle x\rangle,\]
where we can also choose $\gamma=\frac{\gamma_1}{2D}.$

	
	(ii) Now suppose that $\tau>0$ is fixed. Clearly, if  $\omega$ is $(1,\tau)$-thick,  then $\omega$ has a positive measure. 
    
    Assume now that $0< m \leq |\omega|<\infty$. We want to show that \eqref{set-def1} holds with $\rho(x)=\langle x\rangle$, given $\tau$, and some $\gamma<\frac{1}{2}$ and $D>0$. There exists $C$ such that $|\omega\cap(-C,C)|>m/2>0$. We have
    \[|\omega\cap I_{(C+1)\langle x\rangle}(x)|\ge |\omega\cap(-C,C)|\ge \frac{m}{2}.\] We choose $D=C+1$.
    It suffices to show that there exists $\gamma<1/2$ such that $m\ge 2(C+1)\gamma^{\langle x\rangle^\tau}\langle x\rangle$ for any $x$. Note that $\langle x\rangle\ge 1$ for any $x$. Thus, we need to show that for some $\gamma\in(0,1/2)$, such that $\gamma^tt^N\le\frac{m}{2(C+1)}$, where $N=\tau^{-1}$ and $t=\langle x\rangle^\tau$.
    We consider the function $g(t)=\gamma^t t^N$. Then a simple computation gives $g'(t)=g(t)(\log \gamma+Nt^{-1})$, and
    \[\max_{[1,\infty)}g(t)=g(N(\log\gamma^{-1})^{-1})=\exp(N\log N-N\log\log\gamma^{-1}-N).\]
    By choosing $\gamma$ small enough, we  get $g(t)\le \frac{m}{2(C+1)}$ for any $t\ge 1$. If $|\omega|=\infty$, we can still perform the previous arguments by letting $m=n_0$ for some large constant $n_0$.

\begin{remark}
  Through the proof of Lemma \ref{l:set-n} and (i) in Lemma \ref{l:set-nn}, we can choose $\gamma_1$ such that $|\log \gamma|=C|\log \gamma_1|$ for some $C$ depending on $s, a, D, \tau$ in these cases. 
  \label{gamma-g}
\end{remark}

\end{proof}


\section{Spectral inequalities for potentials of polynomial growth}\label{s:two2}
\subsection{Proof of Theorem \ref{th1}}
We first consider the case $S<1$. Then Lemma \ref{l:set-n} provides a partition of the real line into intervals $I_n=(x_n, x_{n+1})$, where $x_n=an^{1/(1-s)}$ and $l_n=|I_n|=|I_{-n}|$ for $n\geq 1$. Note that (\ref{last-1}) and (\ref{haha-100}) imply $l_n\approx |n|^{s/(1-s)}$.
For each $n$, we define the following squares 
\begin{equation*}
    \Omega_{1, n}=I_n\times[-l_n/2, l_n/2],\quad \Omega_{2,n}=4I_n\times[-2l_n,2l_n].
    \end{equation*}

We fix $n$ and denote by $x_n^*=\frac{1}{2}(x_n+x_{n+1})$ the center of $I_n$. Then we consider a rescaled version of $u$  in $[-5,5]\times[-5,5]$
\begin{align*}
    \tilde{u}(x,y)=u(x_n^*+l_nx/2, l_n y/2).
\end{align*} Then, the images of $\mathbb{B}_2$ and $\mathbb{B}_4$ under the transformation $G(x,y)=( x_n^*+l_nx/2, l_ny/2)$ satisfy
\begin{equation}\label{eq-Gset}
G(\mathbb{B}_2)\supset G([-1,1]\times[-1,1])=\Omega_{1,n},\quad G(\mathbb{B}_4)\subset G([-4,4]\times[-4,4])=\Omega_{2,n},
\end{equation}
where $u$ is defined in (\ref{uuu1}).
We also note that $\tilde{u}$ solves  the equation 
\begin{align*}
  \Delta \tilde{u}+\frac{l_n^2}{4} \tilde{V}(x) \tilde{u}=0, \quad \mbox{in} \ [-5, 5]\times [-5, 5],
\end{align*}
where $\tilde{V}(x)= V(x_n^*+l_nx/2)$.
From the assumption \eqref{V.growth} on $V(x)$ and the inequality $l_n=|I_n|\le c_1 x_n^s$ for $s<1$, we get
\begin{equation*}
    |\tilde{V}(x)|\le |V(x_n^*+l_nx/2)|\leq c_2 ( |x^*_{n}+3l_n |+1)^{\beta_2} 
    \leq 2c_3x_n^{\beta_2},
\end{equation*}
for $x\in[-5,5]$. Since $x_n^*\approx x_n$ and $l_n\approx x_n^s$, we get
\begin{align}
    |l_n^2 \tilde{V}(x)|\leq c_4 x_n^{2s+\beta_2}.
    \label{VVV-1}
\end{align}
Let $\omega_n=\omega \cap I_n$ and set $\tilde{\omega}=|I_n|^{-1}(\omega \cap I_n-x^*_n)\subset[-1/2,1/2]$.  By Lemma \ref{l:set-n},  we have $|\tilde{\omega}|\ge \gamma_1^{\langle x_n\rangle^\tau}.$

We split the proof in the following cases: $-\frac{\beta_2}{2}<s<0$, $0\leq s<1$, 
$s\leq  -\frac{\beta_2}{2}$, and $s=1$.

\textbf{Cases 1 and 2 } ($-\frac{\beta_2}{2}<s$):
For the first two cases, we start with similar arguments.
 Lemma \ref{propation-1} implies 
    \begin{align*}
\|\tilde{u}\|_{L^2(\mathbb B_2)}\leq e^{C x_n^{s+\frac{\beta_2}{2}} }\|\tilde{u}\|^{\alpha}_{L^2(\tilde{\omega})}  \|\tilde{u}\|^{1-\alpha}_{L^2(\mathbb B_4)},
\end{align*}
where $\alpha=\frac{1}{C+C\log\frac{1}{|\tilde{\omega}|}}$. Recall the equivalent definition of the thick sets \eqref{set-def1}, i.e (\ref{set-n}) in Lemma \ref{l:set-n}, and Remark \ref{gamma-g}. We get
\begin{align}
    \alpha_n=\frac{1}{ C(1+ \langle x_n\rangle^\tau |\log \gamma|)}.
    \label{set-alpha}
\end{align}
Rescaling back and using the inclusions \eqref{eq-Gset},  we conclude that the following estimate  holds
  \begin{align*}
  \|{u}\|_{L^2(\Omega_{1, n})}\leq l_n^{\alpha_n/2} e^{C x_n^{s+\frac{\beta_2}{2}} }\|{u}\|^{\alpha_n}_{L^2(\omega_n)}  \|{u}\|^{1-\alpha_n}_{L^2(\Omega_{2, n})}.
\end{align*}
As in Section \ref{ss:twow}, we apply  Young's inequality, $AB\leq (1-\alpha_n) A^\frac{1}{1-\alpha_n }+ \alpha_n B^\frac{1}{\alpha_n }$, to obtain 
 \begin{align}
  \|{u}\|_{L^2(\Omega_{1, n})}\leq e^{C x_n^{\sigma}}\left(\frac{\alpha_n l_n^{1/2}}{\delta} \|{u}\|_{L^2(\omega_n)}  +(1-\alpha_n) \delta^{\frac{\alpha_n}{1-\alpha_n} } \| {u}\|_{L^2(\Omega_{3, n})}\right),
  \label{young-help}
\end{align}
where $\sigma=s+\frac{\beta_2}{2}>0$, and $\delta>0$ is arbitrary and will be specified later.

Let $\hat{I}_\lambda=[-C\lambda^{\frac{1}{\beta_1}},\ C\lambda^{\frac{1}{\beta_1}} ],$ where $\lambda>\lambda_1$. Since $u(\cdot, y)\in Ran(P_\lambda(H))$, the estimate (\ref{L2}) implies that
\begin{align}
\|u\|^2_{L^2(\mathbb R)}\leq 2 \|u\|^2_{L^2(-\hat{C}\lambda^{1/\beta_1}, \hat{C}\lambda^{1/\beta_1})}.
\label{new-one}
\end{align}
We introduce  the set    $\mathcal{I}_\lambda=\{n \in \mathbb Z|I_n\cap \hat{I}_\lambda\not=\emptyset \},$
and define $\hat{n}=\max \{|n||n\in \mathcal{I}_\lambda\}.$ For $n\in \mathcal{I}_\lambda$,  we get 
\begin{align}
|x_n|\leq C\lambda^{\frac{1}{\beta_1}},\quad |x_{\hat{n}}|\approx \lambda^{\frac{1}{\beta_1}},\quad |l_{\hat{n}}|\approx \lambda^{\frac{s}{\beta_1}}.  
\label{nnn-1}
\end{align}
 It follows from (\ref{young-help}) that
 \begin{align}
  l_n^{-1/2}\|{u}\|_{L^2(\Omega_{1, n})}\leq e^{C \lambda^{\frac{s}{\beta_1}+\frac{\beta_2}{2\beta_1}} } \left(\frac{1}{\delta} \|{u}\|_{L^2(\omega_n)}  + l_n^{-1/2}\delta^{\frac{\alpha_n}{1-\alpha_n} } \|{u}\|_{L^2(\Omega_{2, n})}\right).
  \label{young}
\end{align}
Since $\delta^{\frac{\alpha_n}{1-\alpha_n} }$ is non-increasing with respect to $\alpha_n$ and $\alpha_n> \alpha_{\hat{n}}$, we can choose $\alpha_n=\alpha_{\hat{n}}$ in the last inequality. Applying (\ref{set-alpha}) and the bounds of $x_{\hat{n}}$ in (\ref{nnn-1}), we obtain
\begin{align}
 \alpha_{\hat{n}}\ge C\langle \lambda\rangle^{\frac{-\tau}{\beta_1}}\frac{1}{|\log \gamma|}.
    \label{alpha-n}
\end{align}
 Taking square of both sides of (\ref{young}) and summing  over $n\in \mathcal{I}_\lambda$, we get
 \begin{align}
  \sum_{n\in \mathcal{I}_\lambda }l_n^{-1}\|{u}\|^2_{L^2(\Omega_{1, n})}\leq 2e^{2C \lambda^{\frac{s}{\beta_1}+\frac{\beta_2}{2\beta_1}} }   \left(\sum_{n\in \mathcal{I}_\lambda } \frac{1}{\delta^2} \|{u}\|^2_{L^2(\omega_n)}  +  \sum_{n\in \mathcal{I}_\lambda } l_n^{-1}\delta^{\frac{2\alpha_{\hat{n}}}{1-\alpha_{\hat{n}}} } \|{u}\|^2_{L^2(\Omega_{2, n})}\right).
  \label{young-1}
\end{align}

We will incorporate the second term in the right hand side of (\ref{young-1}) into the left hand side.
We first provide a lower bound for the left hand side of (\ref{young-1}). WE split the arguments into cases 1 and 2.

\textbf{Case 1 } ($-\frac{\beta_2}{2}<s<0$):
Since $ -\frac{\beta_2}{2}<s<0$, the length of interval $|I_n|=l_n$ is decreasing with respect to $n$, $|I_1|=2a$, and $l_n\ge l_{\hat{n}}\ge c\lambda^{s/\beta_1}$. Clearly 
$\hat{I}_\lambda\subset \cup_{n\in \mathcal{I}_\lambda} I_{n}$.  We have
 \begin{align}
  \sum_{n\in \mathcal{I}_\lambda} l_n^{-1}\|{u}\|^2_{L^2(\Omega_{1, n})}&= \sum_{n\in \mathcal{I}_\lambda}l_n^{-1}\int^{|I_n|/2}_{-|I_n|/2}\int_{I_n}| \sum_{\lambda_k\leq \lambda}  e_k \cosh(\sqrt{\lambda_k} y)\phi_k(x)|^2\, dxdy \nonumber \\ 
  &\ge \frac{2}{a} \sum_{n\in \mathcal{I}_\lambda}\int^{|I_{{n}}|/2}_{0}\int_{I_{n}}| \sum_{\lambda_k\leq \lambda}  e_k \cosh(\sqrt{\lambda_k} y)\phi_k(x)|^2\, dxdy \nonumber \\ 
  &\geq \frac{2}{a}\int^{|I_{\hat{n}}|/2}_{0}\int_{\hat{I}_\lambda} |\sum_{\lambda_k\leq \lambda}  e_k \cosh(\sqrt{\lambda_k} y)\phi_k(x)|^2\, dxdy \nonumber \\
  &\geq\frac{1}{a} \int^{|I_{\hat{n}}|/2}_{0}\int_{\mathbb R} |\sum_{\lambda_k\leq \lambda}  e_k \cosh(\sqrt{\lambda_k} y)\phi_k(x)|^2\, dxdy \nonumber \\
  &=\frac{1}{a}\sum_{\lambda_k\le\lambda}|e_k|^2\int_0^{|I_{\hat{n}}|/2}\cosh^2(\sqrt{\lambda_k}y)dy\nonumber\\
    &\geq C\lambda^{\frac{s}{\beta_1}}\|\phi\|^2_{L^2(\mathbb R)},
    \label{lhs-e-1}
\end{align}
where we used the estimate (\ref{new-one}), the orthogonality of $\phi_k$ in $L^2(\mathbb R)$, and the inequality $\cosh y\ge 1$.

We also need to show an upper bound for the second term in the right hand side of (\ref{young-1}). 
Note that the multiplicity of the overlaps of  $\{4I_{n}\}$ is finite (see Remark \ref{r:overlap}). Then we  obtain an upper bound as follows,
\begin{align}
    \sum_{n\in \mathcal{I}_\lambda} l_n^{-1}\|{u}\|^2_{L^2(\Omega_{2, n})} 
&= \sum_{n\in \mathcal{I}}l_n^{-1}\int^{2|I_n|}_{-2|I_n|}\int_{4I_{n}}| \sum_{\lambda_k\leq \lambda}  e_k \cosh(\sqrt{\lambda_k} y)\phi_k(x)|^2\, dxdy \nonumber \\ 
&\leq \frac{1}{l_{\hat{n}}}\int^{4a}_{-4a}\sum_{n\in \mathcal{I}}\int_{4I_{n}}| \sum_{\lambda_k\leq \lambda}  e_k \cosh(\sqrt{\lambda_k} y)\phi_k(x)|^2\, dxdy \nonumber \\ 
&\leq \frac{C}{l_{\hat{n}}}\int^{4a}_{-4a}\int_{\mathbb R}| \sum_{\lambda_k\leq \lambda}  e_k \cosh(\sqrt{\lambda_k} y)\phi_k(x)|^2\, dxdy \nonumber \\ 
&\leq C\lambda^{-\frac{s}{\beta_1}}e^{C\sqrt{\lambda}} \|\phi\|^2_{L^2(\mathbb R)}.
\label{rhs-e-1}
\end{align}

Then, combining (\ref{young-1}), (\ref{lhs-e-1}), and (\ref{rhs-e-1}), we get
\begin{align*}
    c\lambda^{\frac{s}{\beta_1}}\|\phi\|^2_{L^2(\mathbb R)}\leq e^{C \lambda^{\frac{s}{\beta_1}+\frac{\beta_2}{2\beta_1}} } \frac{1}{\delta^2}  \|\phi\|^2_{L^2(\omega)}+\lambda^{-\frac{s}{\beta_1}}e^{C \lambda^{\frac{s}{\beta_1}+\frac{\beta_2}{2\beta_1}} }e^{C\lambda^{\frac{1}{2}}}\delta^{\frac{2\alpha_{\hat{n}}}{1-\alpha_{\hat{n}}}}  \|\phi\|^2_{L^2(\mathbb R)}.
\end{align*}

Next we want to compare $\frac{2s+\beta_2}{2\beta_1}$ and $\frac{1}{2}$, so split the discussions into two sub-cases.

\textbf{Sub-case 1.1} $(\frac{\beta_1-\beta_2}{2}\leq s<0)$: In this case, we have  $\frac{s}{\beta_1}+\frac{\beta_2}{2\beta_1}\geq \frac{1}{2}$. Then
we have
\begin{align}\label{eq:1234}
    \|\phi\|_{L^2(\mathbb R)}\leq Ce^{C \lambda^{\frac{s}{\beta_1}+\frac{\beta_2}{2\beta_1}} } \frac{1}{\delta}  \|\phi\|_{L^2(\omega)}+Ce^{C \lambda^{\frac{s}{\beta_1}+\frac{\beta_2}{2\beta_1}} }\delta^{\frac{\alpha_{\hat{n}}}{1-\alpha_{\hat{n}}}}  \|\phi\|_{L^2(\mathbb R)},
\end{align}
where we also incorporated $\lambda^{-\frac{s}{\beta_1}}$ and $\lambda^{-\frac{2s}{\beta_1}}$ into the exponential term by using  the inequality $\lambda^{A}\le e^{2A\sqrt{\lambda}}$ for $\lambda>0$. Note the the constant $C$ in the last inequality  depends on $\beta_1,\beta_2,$ and $s$ but not $\lambda$. 
We choose $\delta$ such that
\begin{align*}
    \delta^{\frac{\alpha_{\hat{n}}}{1-\alpha_{\hat{n}}}}
    =\frac{1}{2C} e^{-C \lambda^{\frac{s}{\beta_1}+\frac{\beta_2}{2\beta_1}}}.
\end{align*}
That is, 
\begin{align*}
    \frac{1}{\delta}= (2C)^{\frac{1-\alpha_{\hat{n}}}{\alpha_{\hat{n}}}} e^{ \frac{C(1-\alpha_{\hat{n}})}{\alpha_{\hat{n}}}\lambda^{ \frac{s}{\beta_1}+\frac{\beta_2}{2\beta_1}}}.
\end{align*}
Then the second term  in the right hand side of  (\ref{eq:1234}) can be incorporated into the left hand side.
Using (\ref{eq:1234}), (\ref{nnn-1}), and the value of $\alpha_{\hat{n}}$ in (\ref{alpha-n}), we have
\begin{align}
     \|\phi\|_{L^2(\mathbb R)}&\leq (2C)^{\frac{1-\alpha_{\hat{n}}}{\alpha_{\hat{n}}}}e^{C\frac{1}{\alpha_{\hat{n}}} \lambda^{\frac{s}{\beta_1}+\frac{\beta_2}{2\beta_1}} }\|\phi\|_{L^2(\omega)} \leq e^{\tilde{C} \lambda^{\frac{\tau}{\beta_1}+ \frac{s}{\beta_1}+\frac{\beta_2}{2\beta_1}} }\|\phi\|_{L^2(\omega)},
     \label{conclude-1}
    \end{align}
where $\tilde{C}=C|\log \gamma|.$

\textbf{Sub-case 1.2} $(\frac{-\beta_2}{2} <s<\frac{\beta_1-\beta_2}{2})$: This range of $s$ gives that $0<\frac{s}{\beta_1}+\frac{\beta_2}{2\beta_1}< \frac{1}{2}$. Therefore, we get
\begin{align}
    \|\phi\|_{L^2(\mathbb R)}\leq Ce^{C \lambda^{\frac{s}{\beta_1}+\frac{\beta_2}{2\beta_1}} } \frac{1}{\delta}  \|\phi\|_{L^2(\omega)}+Ce^{C\lambda^{\frac{1}{2}}}\delta^{\frac{\alpha_{\hat{n}}}{1-\alpha_{\hat{n}}}}  \|\phi\|_{L^2(\mathbb R)},
    \label{sub-dis}
\end{align}
where the factors $\lambda^{-\frac{s}{\beta_1}}$ and $\lambda^{-\frac{2s}{\beta_1}}$ are once again absorbed by the exponentials.
We choose $\delta$ such that
$ \delta^{\frac{\alpha_{\hat{n}}}{1-\alpha_{\hat{n}}}}
    =\frac{1}{2C} e^{-C {\lambda}^{\frac{1}{2}}}$
That is, 
    ${\delta}^{-1}= 2^{\frac{1-\alpha_{\hat{n}}}{\alpha_{\hat{n}}}} e^{ \frac{C(1-\alpha_{\hat{n}})}{\alpha_{\hat{n}}}{\lambda}^{\frac{1}{2}}}.$
Then the second term  in the right hand side of  (\ref{sub-dis}) can be incorporated into the left hand side.
Hence (\ref{sub-dis}) and  (\ref{alpha-n}) imply
\begin{align}
     \|\phi\|_{L^2(\mathbb R)}&\leq Ce^{C \lambda^{\frac{s}{\beta_1}+\frac{\beta_2}{2\beta_1}} } e^{\frac{C}{\alpha_{\hat{n}}} {\lambda}^{\frac{1}{2}} }\|\phi\|_{L^2(\omega)} \leq e^{\tilde{C} \lambda^{\frac{\tau}{\beta_1}+\frac{1}{2} }}\|\phi\|_{L^2(\omega)},
     \label{conclude-2}
    \end{align}
where we used the fact that $\frac{s}{\beta_1}+\frac{\beta_2}{2\beta_1}< \frac{1}{2}$, $\tau\geq 0$ and $\tilde{C}=C|\log \gamma|$.
    This ends the discussions in  the Case 1.

\textbf{Case 2} ($0\leq s< 1$): By the arguments in the previous case, we can still obtain (\ref{young-1}). Now we estimate  $\sum_{n\in \mathcal{I}_\lambda} l_n^{-1} \|{u}\|^2_{L^2(\Omega_{1, n})}$ and  $\sum_{n\in \mathcal{I}_\lambda} l_n^{-1}\|{u}\|^2_{L^2(\Omega_{2, n})}$ differently.

Due to the facts that $|I_n|$ is non-decreasing for $0\leq s\leq 1$, and that covering by $\{2I_n\}$ has finite overlaps,  we have
\begin{align}
    \sum_{n\in \mathcal{I}_\lambda} l_n^{-1}\|{u}\|^2_{L^2(\Omega_{2, n})} 
&= \sum_{n\in \mathcal{I}_\lambda} l_n^{-1}\int^{2|I_n|}_{-2|I_n|}\int_{4I_{n}}| \sum_{\lambda_k\leq \lambda}  e_k \cosh(\sqrt{\lambda_k} y)\phi_k(x)|^2\, dxdy \nonumber \\ 
&\leq C a\int^{2|I_{\hat{n}|}}_{0}\int_{\mathbb R}| \sum_{\lambda_k\leq \lambda}  e_k \cosh(\sqrt{\lambda_k} y)\phi_k(x)|^2\, dxdy \nonumber \\ 
&\leq C a\lambda^{\frac{s}{\beta_1}}e^{C\lambda^{\frac{1}{2}+\frac{s}{\beta_1}}} \|\phi\|^2_{L^2(\mathbb R)}, 
\label{rhs-e}
\end{align}
where we used the orthogonality  $\phi_k$ in $L^2(\mathbb R)$ and the estimate $|I_{\hat{n}}|\leq C\lambda^{\frac{s}{\beta_1}}.$
Similarly, we obtain
 \begin{align}
  \sum_{n\in \mathcal{I}_\lambda} l_n^{-1}\|{u}\|^2_{L^2(\Omega_{1, n})}&= \sum_{n\in \mathcal{I}_\lambda}  l_n^{-1}\int^{|I_n|/2}_{-|I_n|/2}\int_{I_{n}}| \sum_{\lambda_k\leq \lambda}  e_k \cosh(\sqrt{\lambda_k} y)\phi_k(x)|^2\, dxdy \nonumber \\ 
  &\geq2 \lambda^{\frac{-s}{\beta}}\sum_{n\in \mathcal{I}_\lambda}\int^{|I_1|/2}_{0}\int_{I_{n}}| \sum_{\lambda_k\leq \lambda}  e_k \cosh(\sqrt{\lambda_k} y)\phi_k(x)|^2\, dxdy \nonumber \\ 
  &\geq 2\int^{|I_1|/2}_{0}\int_{\hat{I}_\lambda} |\sum_{\lambda_k\leq \lambda}  e_k \cosh(\sqrt{\lambda_k} y)\phi_k(x)|^2\, dxdy \nonumber \\
  &\geq C\lambda^{\frac{-s}{\beta}} \int^{|I_1|/2}_{0}\int_{\mathbb R} |\sum_{\lambda_k\leq \lambda}  e_k \cosh(\sqrt{\lambda_k} y)\phi_k(x)|^2\, dxdy \nonumber \\
    &\geq C \lambda^{\frac{-s}{\beta}} \|\phi\|^2_{L^2(\mathbb R)},
    \label{lhs-e}
\end{align}
where we used the estimate (\ref{new-one}). 
Taking  the estimates (\ref{rhs-e}) and (\ref{lhs-e}) into account, we get
\begin{align*}
    \|\phi\|_{L^2(\mathbb R)}\leq Ce^{C \lambda^{\frac{s}{\beta_1}+\frac{\beta_2}{2\beta_1}} } \frac{1}{\delta}  \|\phi\|_{L^2(\omega)}+Ce^{C \lambda^{\frac{s}{\beta_1}+\frac{\beta_2}{2\beta_1}} } \delta^{\frac{\alpha_{\hat{n}}}{1-\alpha_{\hat{n}}}}  \|\phi\|_{L^2(\mathbb R)},
\end{align*}
where we used the fact that $\frac{s}{\beta_1}+\frac{\beta_2}{2\beta_1}\geq \frac{s}{\beta_1}+\frac{1}{2} $ since $\beta_2\geq \beta_1$ and incorporated the polynomial terms into the exponential.
We choose $\delta$ such that
\begin{align*}
    \delta^{\frac{\alpha_{\hat{n}}}{1-\alpha_{\hat{n}}}}
    =\frac{1}{2C} e^{-C \lambda^{\frac{s}{\beta_1}+\frac{\beta_2}{2\beta_1}}}.
\end{align*}
 Then, similar to the
the arguments in Case 1 and  using (\ref{alpha-n}), we conclude that
\begin{align*}
     \|\phi\|_{L^2(\mathbb R)}
     & \leq e^{ \tilde{C} \lambda^{\frac{\tau}{\beta_1}+ \frac{s}{\beta_1}+\frac{\beta_2}{2\beta_1}} }\|\phi\|_{L^2(\omega)},
    \end{align*}
  where $\tilde{C}=C|\log \gamma|$.   This completes the proof of Case 2.

\textbf{Case 3} ($s\leq  -\frac{\beta_2}{2}$): In this case, $s+\frac{\beta_2}{2}\leq 0$. Since $x_n\to\infty$ as $n\to \infty$, then $x_n^{s+\frac{\beta_2}{2}}$ is bounded. Hence $|l_n^2 \tilde{V}(x,y)|$ is bounded in (\ref{VVV-1}). We can still apply Lemma \ref{propation-1} to obtain
    \begin{align*}
  \|\tilde{u}\|_{L^2(\mathbb B_2)}\leq C \|\tilde{u}\|^{\alpha}_{L^2(\tilde{\omega})}  \|\tilde{u}\|^{1-\alpha}_{L^2(\mathbb B_4)}.
\end{align*}
Hence, the rescaling argument gives
\begin{align*}
  \|{u}\|_{L^2(\Omega_{1, n})}\leq C l_n^{\alpha_n/2}\|{u}\|^{\alpha_n}_{L^2(\omega_n)}  \|{u}\|^{1-\alpha_n}_{L^2(\Omega_{2, n})}.
\end{align*}
The same arguments as  in Case 1  lead to 
\begin{align*}
  \sum_{n\in \mathcal{I} }l_n^{-1}\|{u}\|^2_{L^2(\Omega_{2, n})}\leq C   (\sum_{n\in \mathcal{I} } \frac{1}{\delta^2} \|{u}\|^2_{L^2(\omega_n)}  +  \sum_{n\in \mathcal{I} } l_n^{-1}\delta^{\frac{2\alpha_{\hat{n}}}{1-\alpha_{\hat{n}}} } \|{u}\|^2_{L^2(\Omega_{3, n})}).
\end{align*}
For $s<0$, a lower bound of $\sum_{n\in \mathcal{I}_\lambda} l_n^{-1}\|{u}\|^2_{L^2(\Omega_{1, n})}$ follows from (\ref{lhs-e-1}), and, similarly,  an upper bound of $\sum_{n\in \mathcal{I}_\lambda} l_n^{-1}\|{u}\|^2_{L^2(\Omega_{2, n})}$ follows from   (\ref{rhs-e-1}).  Thus, we have
\begin{align}
   \|\phi\|^2_{L^2(\mathbb R)}\leq C \frac{ \lambda^{\frac{-s}{\beta_1}}}{\delta^2}  \|\phi\|^2_{L^2(\omega)}+\lambda^{-\frac{2s}{\beta_1}}e^{C {\lambda}^{\frac{1}{2}} }\delta^{\frac{2\alpha_{\hat{n}}}{1-\alpha_{\hat{n}}}}  \|\phi\|^2_{L^2(\mathbb R)}.
    \label{come-on}
\end{align}
The polynomial factors in the last term into can be absorbed into the exponential.
We choose $\delta$ such that
\begin{align*}
\delta^{\frac{\alpha_{\hat{n}}}{1-\alpha_{\hat{n}}}}
    =\frac{1}{2}e^{-C_1{\lambda}^{\frac{1}{2}}},
\end{align*}
where $C_1$ is large enough. Thus, 
${\delta^{-1}}= e^{C_1 \frac{1-\alpha_{\hat{n}}}{\alpha_{\hat{n}}}{\lambda}^{\frac{1}{2}}}.$
Hence the second term  in the right hand side of  (\ref{come-on}) can be incorporated into the left hand side.
By the value of $\alpha_{\hat{n}}$ in (\ref{alpha-n}), we get that
\begin{align*}
     \|\phi\|_{L^2(\mathbb R)}&\leq e^{C\frac{1}{\alpha_{\hat{n}}} {\lambda}^{\frac{1}{2}} }\|\phi\|_{L^2(\omega)} 
     \leq e^{ \tilde{C} \lambda^{\frac{\tau}{\beta_1}+\frac{1}{2} }}\|\phi\|_{L^2(\omega)},
    \end{align*}
where $\tilde{C}=C|\log \gamma|.$

\textbf{Case 4} ($s=1, \tau=0$): The study of $s=1$ is similar to previous cases. However, now we use $I_n=(-n,n)$ and define $\Omega_{j,n}$, $j=1,2$ as follows
\begin{align*}
    \Omega_{1, n}=I_n\times [-n, \  n], \quad    \Omega_{2, n}=4I_n\times [-4n, \ 4n].
\end{align*}
We fix $\hat{n}$ such that $I_{\hat{n}}\supset I_\lambda$ with $\hat{n}\approx C\lambda^{1/\beta_1}$.
We define
$\tilde{u}(x,y)=u(\hat{n}x,\hat{n} y).$ Then
\begin{align*}
  \Delta \tilde{u}+\hat{n}^2 \tilde{V}(x) \tilde{u}=0, \quad \mbox{in} \ [-5, 5]\times [-5, 5],
\end{align*}
where $\tilde{V}(x)= V(\hat{n}x)$.
From the assumptions on $V(x)$, we get
$    |V(\hat{n}x)|
    \leq C \hat{n}^{\beta_2}.
$
Hence 
\begin{align*}
    |\hat{n}^2 \tilde{V}(x,y)|\leq C \hat{n}^{2+\beta_2}.
\end{align*}
Let $\hat{\omega}_n=\omega \cap {I}_n$. 
Thanks to Lemma \ref{propation-1}, by the rescaling argument, we get
    \begin{align*}
\|\tilde{u}\|_{L^2(\mathbb B_2)}\leq e^{C \hat{n}^{1+\frac{\beta_2}{2}} }\|\tilde{u}\|^{\alpha}_{L^2({\hat{\omega}})}  \|\tilde{u}\|^{1-\alpha}_{L^2(\mathbb B_4)},
\end{align*}
where $\alpha=\frac{1}{C+C\log\frac{1}{|{\hat{\omega}}|}}$. 
Rescaling back to $u$ gives
  \begin{align*}
\|u\|_{L^2(\Omega_{1,\hat{n}})}\leq 
  e^{C \hat{n}^{1+\frac{\beta_2}{2}} }
\hat{n}^{\alpha_/2}
  \|u\|^
  { \alpha }_{L^2(\omega_{\hat{n}})}  \|u\|^{1-\alpha}_{L^2(\Omega_{2,\hat{n}})}.
\end{align*}
From \eqref{(1,0)-thick} and Remark \ref{gamma-g}, we can choose the exponent $\alpha=C|\log \gamma|$. By Young's inequality, we obtain  
 \begin{align}
  \|{u}\|_{L^2(\Omega_{1,\hat{n}})}\leq e^{C \hat{n}^{1+\frac{\beta_2}{2}} } \left(\hat{n}^{1/2}\frac{\alpha}{\delta} \|{u}\|_{L^2(\omega_{\hat{n}})}  +(1-\alpha) \delta^{\frac{\alpha}{1-\alpha} } \| {u}\|_{L^2(\Omega_{2,\hat{n}})}\right)
  \label{young-help-1}
\end{align}
for any $\delta>0$.
We estimate the left-hand side and the second term in the right-hand side in the last inequality as follows,
\begin{align*}
    \|u\|^2_{L^2(\Omega_{1,\hat{n}})}&= \int_{-|I_{\hat{n}}|/2}^{|I_{\hat{n}}|/2}\int_{I_{\hat{n}}}|\sum_{\lambda_k\le\lambda}e_k\phi_k(x)\cosh\sqrt{\lambda_k}y|^2dxdy\\&\ge
    c\int_{-|I_{\hat{n}}|/2}^{|I_{\hat{n}}|/2}\int_{\R}|\sum_{\lambda_k\le\lambda}e_k\phi_k(x)\cosh\sqrt{\lambda_k}y|^2dxdy\\
    &\ge c\hat{n}\|\phi\|^2_{L^2(\R)},
    \end{align*} 
    where we used (\ref{new-one}). We also have
    \begin{align*}
\|u\|^2_{L^2(\Omega_{2,\hat{n}})}&=\int_{-2|I_{\hat{n}}|}^{2|I_{\hat{n}}|}\int_{4I_{\hat{n}}}|\sum_{\lambda_k\le\lambda}e_k\phi_k(x)\cosh\sqrt{\lambda_k}y|^2dxdy\\&\le C \int_{-2|I_{\hat{n}}|}^{2|I_{\hat{n}}|}\int_{\R}|\sum_{\lambda_k\le C\lambda}e_k\phi_k(x)\cosh\sqrt{\lambda_k}y|^2dxdy\\
    &\le C\hat{n}e^{2\hat{n}\sqrt{\lambda}}\|\phi\|^2_{L^2(\R)}.\end{align*}
    Then, substituting the last inequalities in \eqref{young-help-1} and using that $\hat{n}\approx C\lambda^{1/\beta_1}$, we obtain
\begin{equation}\label{eq-s=1}\|\phi\|^2_{L^2(\R)}\le Ce^{C \lambda^{\frac{1}{\beta_1}+\frac{\beta_2}{2\beta_1}} } \pr{\frac{1}{\delta} \|\phi\|_{L^2(\omega_{\hat{n}})}  +\delta^{\frac{\alpha}{1-\alpha} } e^{C\lambda^{\frac12+\frac{1}{\beta_1}}}\| \phi\|_{L^2(\R)}}.\end{equation}
  Clearly, $\frac{\beta_2}{2\beta_1}\ge \frac{1}{2}$.  We choose $\delta$ such that
\begin{align*}
   \delta^{\frac{ {\alpha}}{1- {\alpha}}}
    =\frac{1}{2} e^{-C \lambda^{\frac{1}{\beta_1}+\frac{\beta_2}{2\beta_1}}}.
\end{align*}
Then the second term on the right hand side of \eqref{eq-s=1} is absorbed by the left hand side. We get
\begin{align*}
     \|\phi\|_{L^2(\mathbb R)}
     & \leq e^{\tilde{C} \lambda^{\frac{1}{\beta_1}+\frac{\beta_2}{2\beta_1}} }\|\phi\|_{L^2(\omega)},
    \end{align*}
    where $\tilde{C}=C|\log \gamma|.$
    This completes the proof of Case 4.

Based on the discussions in Case 1.1, Case 2 and Case 4, we conclude that the estimates (\ref{conclude-1}) holds for  $\frac{\beta_1-\beta_2}{2}\leq s\leq 1$. The discussions in  Case 1.2 and Case 3 show that the estimates (\ref{conclude-2}) holds for $s<\frac{\beta_1-\beta_2}{2}$. This finishes the proof of Theorem \ref{th1}.

\subsection{Spectral inequalities on sets of positive measure}\label{s:meas} This section is devoted to the proof of Theorem \ref{th2}. Its proof is similar to the Case 4 in the last Theorem. We include the details for the reader's convenience.

\begin{proof}[Proof of Theorem \ref{th2}]
We may assume that $|\omega|<\infty$ otherwise we take a subset of $\omega$ of finite measure. Since $\lim_{R\to \infty} |\omega\cap I_R|=|\omega|>0$, there exists $R_0$ depending on $\omega$ such that
\begin{align*}
    |\omega\cap I_R|\geq \frac{|\omega|}{2}, \quad \mbox{for} \ R\geq R_0,
\end{align*}
where $I_R=(-R, R).$
 As above, we consider $u(x, y)=\sum_{0<\lambda_k\leq \lambda} e_k \cosh(\sqrt{\lambda_k} y)\phi_k(x)$,
which satisfies $\frac{\partial u}{\partial y}=0$ on $\{\mathbb R^2| y=0\}$, 
\begin{align*}
-\Delta u+V(x)u=0 \quad \mbox{in} \ \mathbb R^2
\end{align*}
and (\ref{new-one}). Let $\hat{C}$ be  as in (\ref{new-one}). We define $\rho(\lambda)=\hat{C}\lambda^{\frac{1}{\beta_1}}$. We may assume that $\lambda$ is large enough and $\rho(\lambda)>R_0$. We fix $\lambda$ and consider 
\begin{align*}\tilde{u}(x,y)=u(\rho(\lambda)x,\rho(\lambda)y). \end{align*} 
Then $\tilde{u}$ satisfies the  equation
$ \Delta\tilde{u}+\rho(\lambda)^2\tilde{V}\tilde{u}=0$,
where $\tilde{V}(x)=V(\rho(\lambda)x )$. The condition \eqref{V.growth} implies
\begin{align}
|\rho(\lambda)^2\tilde{V}(x)|\leq  C\lambda^{\frac{2}{\beta_1}+\frac{\beta_2}{\beta_1}}, \quad {\text{for}}\quad |x|\le 5.
   \label{VVV-2}
\end{align}
   Define $\hat{I}_\lambda=[-\rho(\lambda), \rho(\lambda)]$ and $\tilde{\omega}= |\hat{I}_{\lambda}|^{-1} (\omega\cap \hat{I}_{\lambda})$. By  Lemma \ref{propation-1} and (\ref{VVV-2}), we obtain
  \begin{align*}
  \|\tilde{u}\|_{L^2(\mathbb B_2)}\leq  e^{C \lambda^{\frac{1}{\beta_1}+\frac{\beta_2}{2\beta_1}} }\|\tilde{u}\|^{\alpha}_{L^2(\tilde{\omega})}  \|\tilde{u}\|^{1-\alpha}_{L^2(\mathbb B_4)},
\end{align*}
where $\alpha=\frac{1}{C+C\log \frac{1}{|\tilde{w}|}}$. Our assumption $\rho(\lambda)\ge R_0$ implies that $|\tilde{\omega}|\ge |\omega|(2|\hat{I}_\lambda|)^{-1}$. Rescaling back to $u$ yields 
  \begin{align}
  \|{u}\|_{L^2(\mathbb B_{2\rho(\lambda)})}\leq \lambda^{\frac{\alpha}{\beta_1}} e^{C \lambda^{\frac{1}{\beta_1}+\frac{\beta_2}{2\beta_1}} }\|{u}\|^{\alpha}_{L^2(\omega\cap \hat{I}_\lambda)}  \|{u}\|^{1-\alpha}_{L^2(\mathbb B_{4\rho(\lambda)})},  \label{measure-key}
\end{align}
where $\alpha$ satisfies
\begin{align}
    \alpha\geq  \frac{1}{ C+C\log (\lambda+1)},
  \label{simple-alpha}
\end{align}
where $C$ depends on $\omega, \beta_1$.

We want to estimate the norms $\|{u}\|_{L^2(\mathbb B_{\rho(\lambda)})}$ and  $\|{u}\|_{L^2(\mathbb B_{2\rho(\lambda)})}$. As in Case 4, we have
\begin{align*}
\|{u}\|^2_{L^2(\mathbb B_{4\rho(\lambda)})} 
\leq e^{C \lambda^{\frac{1}{2}+\frac{1}{\beta_1}}}  \|\phi\|^2_{L^2(\mathbb R)}, 
\quad
\|{u}\|^2_{L^2(\mathbb B_{2\rho(\lambda)})}
\geq  2{C} \lambda^{\frac{1}{\beta_1}}\|\phi\|^2_{L^2(\mathbb R)}.
\end{align*}
Therefore (\ref{measure-key}) implies
\begin{align*}
   \|\phi\|_{L^2( \mathbb R)}&\leq e^{C \lambda^{\frac{1}{\beta_1}+\frac{\beta_2}{2\beta_1}} } e^{C \lambda^{\frac{1}{\beta_1}+\frac{1}{2}} }\|\phi\|^{\alpha}_{L^2(\omega\cap \mathbb B_{R_0})}  \|\phi\|^{1-\alpha}_{L^2( \mathbb R)} 
   \leq e^{C \lambda^{\frac{1}{\beta_1}+\frac{\beta_2}{2\beta_1}} }\|\phi\|^{\alpha}_{L^2(\omega\cap \mathbb B_{R_0})}  \|\phi\|^{1-\alpha}_{L^2( \mathbb R)} 
\end{align*}
as $\frac{\beta_2}{2\beta_1}\geq \frac{1}{2}.$ Dividing by $\|\phi\|^{1-\alpha}_{L^2(\R)}$ and raising both sides of the inequality to $\alpha^{-1}$,  we obtain 
\begin{align*}
   \|\phi\|_{L^2( \mathbb R)}\leq e^{\frac{C}{\alpha} \lambda^{\frac{1}{\beta_1}+\frac{\beta_2}{2\beta_1}} } \|\phi\|_{L^2(\omega\cap \mathbb B_{R_0})}    
   \leq e^{C \lambda^{\frac{1}{\beta_1}+\frac{\beta_2}{2\beta_1}}\log(\lambda+1) }\|\phi\|_{L^2(\omega)}.
\end{align*}
we also used (\ref{simple-alpha}) in the last inequality. This completes the proof of Theorem \ref{th2}.
\end{proof}

\section{Appendix}
\subsection{Proof of Lemma \ref{lemma-pro}.}

  Let $cap(\mathcal{E})$ denote the logarithmic capacity of a set $\mathcal{E}\subset\C$. A classical estimate for holomorphic functions gives
\begin{equation}
    \sup_{\mathbb B_2}|h(z)|\leq \sup_{\mathcal{E}}|h(z)|^\alpha \sup_{\mathbb B_4}|h(z)|^{1-\alpha},
   \label{three-hol-1}
\end{equation}
where $\alpha=(C-C\log(cap(\mathcal{E}))^{-1}$ and $C$ is an absolute constant, see e.g. \cite{M04}. We also know that  $cap(\mathcal{E})\ge \frac14|\mathcal{E}|$ if $\mathcal{E}\subset\R$ (see \cite{L}). This gives the statement of Lemma \ref{lemma-pro} with the supremum norm over $\mathcal{E}$.

We want to replace the $\sup_{\mathcal{E}}|h|$ in the inequality (\ref{three-hol-1}) by the $L^2$-norm of $h$ over $\mathcal{E}$. We define
\begin{align*}
    \mathcal{E}_1=\{x\in \mathcal{E}| h(x)\geq \frac{2}{ |\mathcal{E}|^{1/2} }\|h\|_{L^2(\mathcal{E})}\}.
\end{align*}
By the Chebyshev inequality,  
\[|\mathcal{E}_1|\leq \frac{|\mathcal{E}|}{4\|h\|_{L^2(\mathcal{E})}^2}\|h\|_{L^2(\mathcal{E})}^2=\frac{|\mathcal{E}|}{4}.\]  
We consider the complement of $\mathcal{E}_1$ in $\mathcal{E}$,
\begin{align*}
   \mathcal{E}_0=\mathcal{E}\backslash \mathcal{E}_1=\left\{x\in \mathcal{E}| h(x)< \frac{2}{|\mathcal{E}|^{1/2}} \|h\|_{L^2(\mathcal{E})}\right\}.
\end{align*}
Then $|\mathcal{E}_0|\geq  \frac{3|\mathcal{E}|}{4}$. Applying the inequality (\ref{three-hol-1}) with $\mathcal{E}$ replaced by $\mathcal{E}_0$ to get
\begin{align*}
  \sup_{\mathbb B_2}|h|&\leq  \left(\frac{2}{|\mathcal{E}|^{1/2}}\right)^\alpha \|h\|_{L^2(\mathcal{E})}^{\alpha} \sup_{\mathbb B_{4}}|h |^{1-\alpha},
\end{align*}
where $\alpha=(C+C\log\frac{1}{|\mathcal{E}_0|})^{-1}$.
Note that \[1+\log\frac{1}{|\mathcal{E}_0|}\le C_1+\log\frac{1}{|\mathcal{E}|}\le C_1(1+\log\frac{1}{|\mathcal{E}|}).\]
Hence   $(\frac{2}{|\mathcal{E}|^{1/2}})^\alpha \leq C_0$ for some constant $C_0$ independent of $\mathcal{E}$. Thus, 
\begin{align*}
  \sup_{\mathbb B_2}|h|& \leq C\|h\|_{L^2(\mathcal{E})}^{\alpha}\sup_{\mathbb B_{4}}|h|^{1-\alpha},
\end{align*}
with $\alpha=(C+C\log\frac{1}{|\mathcal{E}|})^{-1}$, which completes the proof of the lemma.


\subsection{Proof of Lemma \ref{l:twow}} We follow the argument of \cite{DSV24}, see also \cite{GY12}.
Recall that $\Phi(x)$ is a positive increasing function, $\lim_{x\to\infty} \Phi(x)=+\infty$ and $V(x)\ge \Phi(|x|)$. For $H=-\Delta+V$,
let $N(\lambda)=\#\{\lambda_j\leq \lambda\}$. Then
$$ N(\lambda)\leq \sum_{\lambda_j\leq \lambda}(1+\lambda-\lambda_j)^{1/2}.$$ Thus, by the the Lieb-Thirring inequality in dimension one (see e.g. \cite{W96}), we have
\begin{align}
\label{eq-ids}
   N(\lambda)&\leq \sum_{\lambda_j\leq \lambda+1}(\lambda+1-\lambda_j)^{1/2}\leq C \int_{\mathbb R} \max\{\lambda+1-V(x), 0 \}dx  \nonumber \\
   &\leq \int^{\Phi^{-1}(\lambda+1)}_{-\Phi^{-1}(\lambda+1) } (\lambda+1 )dx
   \leq
   2C(\lambda+1)\Phi^{-1}(\lambda+1),
\end{align}
where $C$ is an absolute constant.

We will need the following Agmon-type estimates, see Lemma 3.2 in \cite{DSV24}.
\begin{lemma} Let $V\in L^1_{loc}(\R)$, $V>1+\mu^2$ for $|x|>R$ and $u\in H^1_{loc}(\R)\cap L^2(\R)$ satisfies 
$-\Delta u+Vu=\phi$. Suppose that $e^{2\mu |x|}\phi(x)\in L^2(\R)$. Then
\[\|e^{\mu|x|}u\|_{L^2(\R)}\le \frac{1}{2}\|e^{2\mu|x|}\phi\|^2_{L^2(\R\setminus \mathbb B(0,R))}+(4\mu+6)e^{2\mu (R+1)}\|u\|_{L^2(\R)}.\]
\end{lemma}
Thanks to this lemma, we have
\begin{corollary} Suppose that $-\Delta \varphi_\lambda+V\varphi_\lambda=\lambda \varphi_\lambda$ and $V>2+\lambda$ for $|x|>R$. Then for $r>1$, we have
\[\|\varphi_\lambda\|_{L^2(\R\setminus(-R-r,R+r))}\le 10 e^{-2r+2}\|\varphi_\lambda\|^2_{L^2(\R)}.\] 
\end{corollary}
\begin{proof} We apply the last  lemma  with $V$ replaced by  $V-\lambda$, $\phi=0$, and $\mu=1$ to obtain
\[\|\varphi_\lambda\|^2_{L^2(\R\setminus(-R-r,R+r))}\le e^{-2(R+r)}\|e^{|x|} \varphi_\lambda(x)\|^2_{L^2(\R)}\le 10 e^{-2r+2}\|\varphi_\lambda\|^2_{L^2(\R)}.\]
\end{proof}

The following Corollary provides the localization estimate for the linear combinations of the eigenfunctions in Lemma \ref{l:twow}.
\begin{corollary} Suppose that $\varphi_\lambda\in Ran(P_\lambda(H))$, $H=-\Delta+V$, and $V>2+\lambda$ for $|x|>R$. We have
\[\|\varphi_\lambda\|^2_{L^2(\R\setminus(-R-r,R+r))}\le \frac{1}{2}\|\varphi_\lambda\|^2_{L^2(\R)},\]
when $e^{2r}>C_0(\lambda+1)\Phi^{-1}(\lambda+1)$. 
\end{corollary}
\begin{proof} We can write $\varphi_\lambda=\sum_{\lambda_k\leq \lambda} (\varphi_\lambda, \varphi_k) \varphi_k$, where $\varphi_k$ is an eigenfunction of $H$ with eigenvalue $\lambda_k\leq \lambda$.
Then, by the last lemma,
\begin{align*}
\|\varphi_\lambda\|^2_{L^2(\R\setminus(-R-r,R+r))}&\le N(\lambda)\sum_{\lambda_k\leq \lambda} (\varphi_\lambda, \varphi_k)^2 \|\varphi_k\|^2_{L^2(\R\setminus(-R-r,R+r))}\\
&\le 10N(\lambda) e^{-2r+2}\sum_k (\varphi_\lambda, \varphi_k)^2\|\varphi_k\|^2_{L^2(\R)}=10N(\lambda) e^{-2r+2}\|\varphi\|^2_{L^2(\R)}.
\end{align*}
We combine the last inequality with the estimate \eqref{eq-ids} on $N(\lambda)$ to obtain the required inequality.
\end{proof}

\bibliography{Mybib}

\begin{thebibliography}{10}

\bibitem{ARRV}
G.~Alessandrini, L.~Rondi, E.~Rosset, and S.~Vessella.
\newblock The stability for the {C}auchy problem for elliptic equations.
\newblock {\em Inverse Problems}, 25, 2009.

\bibitem{BJPS21}
K.~Beauchard, P.~Jaming, and K.~Pravda-Starov.
\newblock Spectral inequality for finite combinations of {H}ermite functions and null-controllability of hypoelliptic quadratic equations.
\newblock {\em Studia Mathematica}, 260(1):1--43, 2021.

\bibitem{BPS18}
K.~Beauchard and K.~Pravda-Starov.
\newblock Null-controllability of hypoelliptic quadratic differential equations.
\newblock {\em Journal de l’{\'E}cole Polytechnique-Math{\'e}matiques}, 5:1--43, 2018.

\bibitem{BM21}
N.~Burq and I.~Moyano.
\newblock Propagation of smallness and spectral estimates.
\newblock {\em arXiv:2109.06654}, 2021.

\bibitem{BM22}
N.~Burq and I.~Moyano.
\newblock Propagation of smallness and control for heat equations.
\newblock {\em Journal of the European Mathematical Society}, 25(4):1349--1377, 2022.

\bibitem{Cor07}
J.-M. Coron.
\newblock {\em Control and nonlinearity}.
\newblock Number 136. American Mathematical Soc., 2007.

\bibitem{DH96}
C.~De~Coster and P.~Habets.
\newblock {\em Upper and Lower Solutions in the Theory of ODE Boundary Value Problems: Classical and Recent Results}, pages 1--78.
\newblock Springer Vienna, Vienna, 1996.

\bibitem{DSV23}
A.~Dicke, A.~Seelmann, and I.~Veseli{\'c}.
\newblock Uncertainty principle for {H}ermite functions and null-controllability with sensor sets of decaying density.
\newblock {\em Journal of Fourier Analysis and Applications}, 29(1):11, 2023.

\bibitem{DSV24}
A.~Dicke, A.~Seelmann, and I.~Veseli{\'c}.
\newblock Spectral inequality with sensor sets of decaying density for {S}chr{\"o}dinger operators with power growth potentials.
\newblock {\em Partial Differential Equations and Applications}, 5(2):7, 2024.

\bibitem{F23}
B.~Foster.
\newblock Results on gradients of harmonic functions on {L}ipschitz surfaces.
\newblock {\em arXiv:2304.11344}, 2023.

\bibitem{GY12}
J.~Gagelman and H.~Yserentant.
\newblock A spectral method for {S}chr\"odinger equations with smooth confinement potentials.
\newblock {\em Numer. Math.}, 122(2):383--398, 2012.

\bibitem{JL99}
D.~Jerison and G.~Lebeau.
\newblock Nodal sets of sums of eigenfunctions. {H}armonic analysis and partial differential equations (chicago, il, 1996), 223--239.
\newblock {\em Chicago Lectures in Math., Univ. Chicago Press, Chicago, IL}, 1999.

\bibitem{KSW15}
C.~Kenig, L.~Silvestre, and J.-N. Wang.
\newblock On {L}andis' conjecture in the plane.
\newblock {\em Comm. Partial Differential Equations}, 40(4):766--789, 2015.

\bibitem{K01}
O.~Kovrijkine.
\newblock Some results related to the {L}ogvinenko-{S}ereda theorem.
\newblock {\em Proceedings of the American Mathematical Society}, 129(10):3037--3047, 2001.

\bibitem{L}
N.~Landkof.
\newblock {\em Foundations of modern potential theory}.
\newblock Berlin, New York, Springer-Verlag, 1972.

\bibitem{LLR22a}
J.~Le~Rousseau, G.~Lebeau, and L.~Robbiano.
\newblock {\em Elliptic {C}arleman Estimates and Applications to Stabilization and Controllability, Volume I}, volume~97.
\newblock Springer, 2022.

\bibitem{LR95}
G.~Lebeau and L.~Robbiano.
\newblock Contr\^ole exact de l'\'equation de la chaleur.
\newblock {\em Communications in Partial Differential Equations}, 20(1-2):335--356, 1995.

\bibitem{L88}
J.-L. Lions.
\newblock {\em Contr\^olabilit\'e{} exacte, perturbations et stabilisation de syst\`emes distribu\'es. {T}ome 1}, volume~8 of {\em Recherches en Math\'ematiques Appliqu\'ees [Research in Applied Mathematics]}.
\newblock Masson, Paris, 1988.
\newblock Contr\^olabilit\'e{} exacte. [Exact controllability], With appendices by E. Zuazua, C. Bardos, G. Lebeau and J. Rauch.

\bibitem{LM18}
A.~Logunov and E.~Malinnikova.
\newblock Quantitative propagation of smallness for solutions of elliptic equations.
\newblock In {\em Proceedings of the International Congress of Mathematicians 2018}, pages 2391--2411. World Scientific, 2018.

\bibitem{M04}
E.~Malinnikova.
\newblock Propagation of smallness for solutions of generalized {C}auchy-{R}iemann systems.
\newblock {\em Proc. Edinburgh Math. Society}, 47:191--204, 2004.

\bibitem{M22}
J.~Martin.
\newblock Spectral inequalities for anisotropic {S}hubin operators.
\newblock {\em arXiv:2205.11868}, 2022.

\bibitem{MPS23}
J.~Martin and K.~Pravda-Starov.
\newblock Spectral inequalities for combinations of {H}ermite functions and null-controllability for evolution equations enjoying {G}elfand--{S}hilov smoothing effects.
\newblock {\em Journal of the Institute of Mathematics of Jussieu}, 22(6):2533--2582, 2023.

\bibitem{M08}
L.~Miller.
\newblock Unique continuation estimates for sums of semiclassical eigenfunctions and null-controllability from cones.
\newblock {\em hal-00411840}, 2008.

\bibitem{NTTV20}
I.~Naki{\'c}, M.~T{\"a}ufer, M.~Tautenhahn, and I.~Veseli{\'c}.
\newblock Sharp estimates and homogenization of the control cost of the heat equation on large domains.
\newblock {\em ESAIM: Control, Optimisation and Calculus of Variations}, 26:54, 2020.

\bibitem{SVW98}
C.~Shubin, R.~Vakilian, and T.~Wolff.
\newblock Some {H}armonic {A}nalysis questions suggested by {A}nderson-{B}ernoulli models.
\newblock {\em GAFA, Geom. Funct. Anal}, 8:932--964, 1998.

\bibitem{SSY23}
P.~Su, C.~Sun, and X.~Yuan.
\newblock Quantitative observability for one-dimensional {S}chr{\"o}dinger equations with potentials.
\newblock {\em Journal of Functional Analysis}, 288(2):110695, 2025.

\bibitem{Wang24}
Y.~Wang.
\newblock Quantitative 2d propagation of smallness and control for 1d heat equations with power growth potentials.
\newblock {\em ESAIM: Control, Optimisation and Calculus of Variations}, 31:24, 2025.

\bibitem{W96}
T.~Weidl.
\newblock On the {L}ieb-{T}hirring constants $l_{\gamma, 1}$ for $\gamma\geq 1/2$.
\newblock {\em Communications in Mathematical Physics}, 178(1):135--146, 1996.

\bibitem{Zhu24}
J.~Zhu.
\newblock Spectral inequalities for {S}chr\"odinger equations with various potentials.
\newblock {\em arXiv:2403.08975}, 2024.

\bibitem{ZZ23}
J.~Zhu and J.~Zhuge.
\newblock Spectral inequality for {S}chr\" odinger equations with power growth potentials.
\newblock {\em arXiv:2301.12338, To appear in Indiana University Mathematics Journal}, 2023.

\bibitem{ZZ24}
J.~Zhu and J.~Zhuge.
\newblock Observability inequalities for heat equations with potentials.
\newblock {\em arXiv:2409.09476}, 2024.

\bibitem{zhu25}
Y.~Zhu.
\newblock Propagation of smallness for solutions of elliptic equations in the plane.
\newblock {\em Mathematics in Engineering}, 7(1):1--12, 2025.

\end{thebibliography}
\bibliographystyle{abbrv}
\end{document}